%% file: BHR_arXiv.tex
\documentclass[11pt]{amsart}

\usepackage{amsmath,amssymb,amsfonts,amsthm,amscd}
\usepackage[left=2.7cm,right=2.7cm,top=2.7cm,bottom=2.7cm]{geometry}
\usepackage{fancyhdr}
\usepackage{setspace}
\usepackage{ragged2e}
\usepackage[T1]{fontenc}
\usepackage[utf8]{inputenc}
\usepackage{eurosym}
\usepackage[english]{babel}
\usepackage{lmodern}
\usepackage{enumitem}
\usepackage{caption}
\usepackage{subcaption}
\usepackage{float}
\usepackage{graphicx}
\usepackage[export]{adjustbox}
\usepackage{cleveref}
\usepackage{pdfpages}
\usepackage{cancel}

\input{settings.tex}

\def \R {{\Bbb R}}

\def \s {{\Bbb S}}
\def \B {{\Bbb B}}

\newtheorem{theorem}{Theorem}[section]
\newtheorem*{th*}{Theorem~\ref{th:cat}}
\newtheorem*{th**}{Theorem~\ref{th:k}}
\newtheorem{proposition}[theorem]{Proposition}
\newtheorem{corollary}[theorem]{Corollary}
\newtheorem{lemma}[theorem]{Lemma}
\newtheorem{claim}[theorem]{Claim}
\newtheorem{remark}[theorem]{Remark}

\title{Tensile minimal surfaces and thread boundary problems}

\thanks{The third author was partially supported by the IMAG–Maria de Maeztu grant CEX2020-001105-M / AEI / 10.13039/501100011033, MINECO/MICINN/FEDER grant no. PID2023-150727NB-I00 and Junta de Andalucia grant P18-FR4049.}

\author{Romane Boutillier}
\address{Laboratoire Navier, École nationale des ponts et chaussées, Univ. Gustave Eiffel, CNRS, Champs-sur-Marne, France}
\email{romane.boutillier@univ-eiffel.fr}

\author{Laurent Hauswirth}
\address{Laboratoire d’Analyse et de Mathématiques Appliquées UMR8050, Univ. Gustave Eiffel, Champs-sur-Marne, France}
\email{laurent.hauswirth@univ-eiffel.fr}

\author{Magdalena Rodríguez}
\address{Departamento de Geometría y Topología, Universidad de Granada, Spain}
\email{magdarp@ugr.es}

\date{}

\begin{document}

\begin{abstract}
Minimal surfaces bounded by cables or threads arise naturally in tensile architecture: a membrane under uniform tension takes the shape of a minimal surface, and its
flexible, inextensible boundary lies along an asymptotic line of constant geodesic
curvature. Such configurations have been studied experimentally since the 1960s at
the Institute for Lightweight Structures in Stuttgart and more recently realized in
gridshells built along networks of asymptotic and geodesic curves. Motivated by this
architectural context, we construct two new families of embedded minimal surfaces
bounded by finitely many asymptotic arcs of constant curvature, using the
Plateau-conjugate method applied to the solution of a partially-free boundary problem for minimal disks that meet the unit sphere orthogonally along the free boundary component. 

For every integer $m \geq 3$, we prove the
existence of a one‑parameter family of embedded minimal annuli, called tensile
catenoids, whose boundary components each consist of $m$ asymptotic arcs of constant
curvature that meet at cusps. As the parameter varies, the family degenerates from a
planar configuration to a union of $m$ minimal disks. 

For every integer $k\geq 3$, we prove the existence of an embedded minimal surface with the topology of a sphere minus $k$
disks, called a tensile $k$‑noid, each of whose $k$ boundary components consists of four
asymptotic arcs joined by cusps. 

Both families are symmetric with respect to a horizontal plane and possess several vertical planes of symmetry. Embeddedness
follows from showing that the fundamental piece obtained by conjugation is a graph contained in the region delimited by the planes of symmetry.
\end{abstract}

\maketitle

\renewcommand{\thepage}{\arabic{page}}

\section{Introduction}

In architecture, minimal surfaces belong to the structural family of tensile surfaces. They developed in the 20th century \cite{berger_light_1996}, particularly through the work of Frei Otto (one of his famous large-scale projects is the roof of the Munich stadium, shown in \cref{fig:munich}). 
Certain aspects of membrane structures explain their popularity in architecture: they cover large spans, are suitable for temporary installations due to their ease of assembly and disassembly, and are lightweight and inexpensive.

\begin{figure}[hbt!]
\centering
\includegraphics[width = 0.5\textwidth]{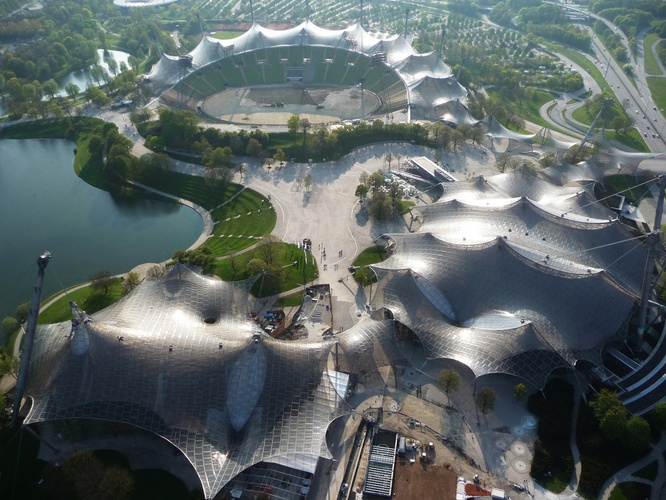}
\caption{Munich Olympic Stadium. Image © Atelier Frei Otto Warmbronn}
\label{fig:munich}
\end{figure}

The study and exploration of minimal surfaces was used as a design tool for architects and engineers, especially at the Institute for Lightweight Structures (IL). This renowned structural research and training center was founded in 1964 in Stuttgart by Fritz Leonhardt, with Frei Otto as its first director. IL architects and engineers have extensively studied and modeled minimal surfaces using small-scale soap films (see \cref{fig:soapILEK}) and wire models (see \cref{fig:cablesILEK}). In the design process, soap models are used to generate the geometry (measured using photogrammetric methods), while wire models are used to perform simulations and mechanical measurements (in particular, measuring the tension in the cables). Some studies have eventually resulted in the construction of structures, such as the IL building, constructed in 1966 in Stuttgart (see \cref{fig:ILEKaerial}). Many pictures and models can be found in the literature \cite{otto_tensile_1973, DenkenModellen, IL18, Otto_Moma, vrachliotis_maquette_2020}. Minimal surfaces were explored and built for pavilions, airports, stadiums, and sculptures (for more information on the typologies of tensile surfaces, the mechanical constraints, and the practical details of the implementation of these structures, the reader is invited to refer to \cite{schlaich_tensile_1989}).

\begin{figure}[htb!]
\centering
\begin{subfigure}[t][][b]{0.24\textwidth}
\centering
    \includegraphics[height = 4cm]{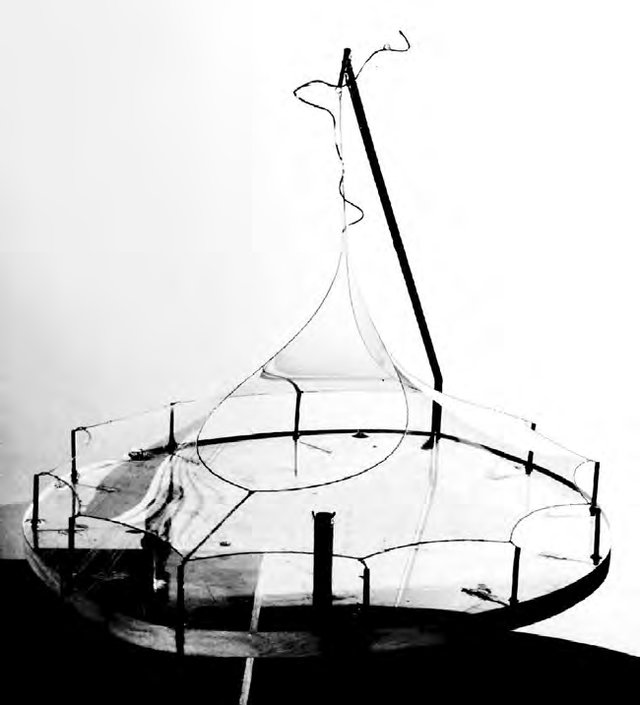}
    \caption{Soap film model of the IL, photo IL archive.}
    \label{fig:soapILEK}
\end{subfigure}
\hfill
\begin{subfigure}[t][][b]{0.37\textwidth}
\centering
    \includegraphics[height = 4cm]{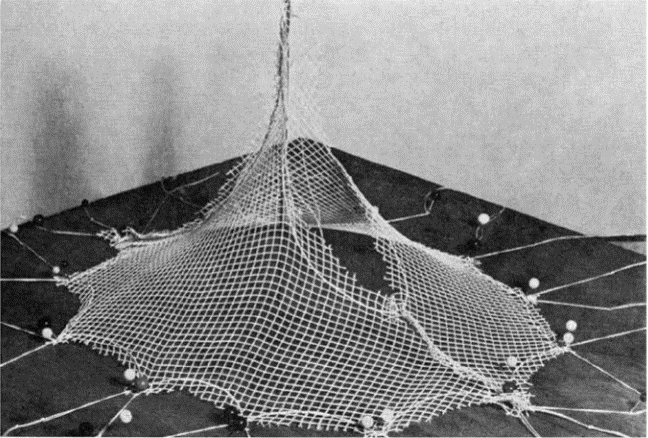}
    \caption{Wire model of the IL, photo \cite{leonhardt_deutsche_1968}.}
    \label{fig:cablesILEK}
\end{subfigure}
\hfill
\begin{subfigure}[t][][b]{0.35\textwidth}
\centering
    \includegraphics[height = 4cm]{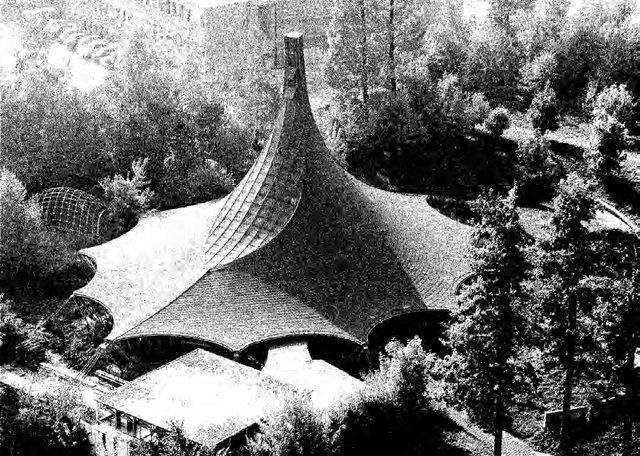}
    \caption{IL, photo IL archive.}
    \label{fig:ILEKaerial}
\end{subfigure}
\caption{Study and building of the Institute for Lightweight Structures.}
\label{fig:ILEK}
\end{figure}

A variety of shapes were created, with several boundary conditions (rigid, flexible, free boundaries). In their book focusing on bubbles \cite{IL18}, the IL team cites the work of Hildebrandt and Nitsche, and modelled their sketches \cite{dierkes_minimal_2010} with soap films (\cref{fig:freeBoundaryIL}).

\begin{figure}[h]
\centering
    \includegraphics[width = 0.3\textwidth]{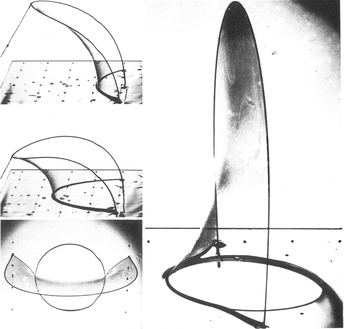}
    \caption{Free boundary soap models, source \cite{IL18}}
    \label{fig:freeBoundaryIL}
\end{figure}

Simple mechanical calculations (the equilibrium of a soap film and its boundary) allow us to recover two properties of this variational problem: namely, that a surface uniformly stretched in all directions is a minimal surface, and that an inextensible wire bounding a minimal surface is an asymptotic line of constant geodesic curvature \cite{dierkes_minimal_1992-2,sehlstrom_tensioned_2021}.

\paragraph{\textbf{Forces in minimal surfaces}}

The stresses in membrane structures lie within the tangent plane: compression, tension, shear, but no bending. In certain force configurations, these structures take the shape of minimal surfaces. Indeed, the equilibrium of a portion of surface, as illustrated in~\cref{Fig:pressure}, in the normal direction, is given by the commonly known Young-Laplace equation. It relates the tangential forces in an inextensible surface to the surface curvature. It is given in a local conformal parametrisation by \cref{eq:eq3}, where $\lambda$ is the conformal factor, $T_{11}$ and $T_{22}$ are the normal forces on the surface boundary, $T_{12}$ is the tangential force on the surface boundary, $p$ is the applied normal load, and $e, f, g$ are the coefficients of the second fundamental form. A detailed demonstration can be found in \cite{frey_analyse_2003} or \cite{ventsel_thin_2001}.

\begin{align}
    T_{11}e+2T_{12}f+T_{22}g = p \lambda^2 \label{eq:eq3}
\end{align}

\begin{figure}[!htb]
     \centering
     \begin{subfigure}{0.45\textwidth}
         \includegraphics[width=\linewidth]{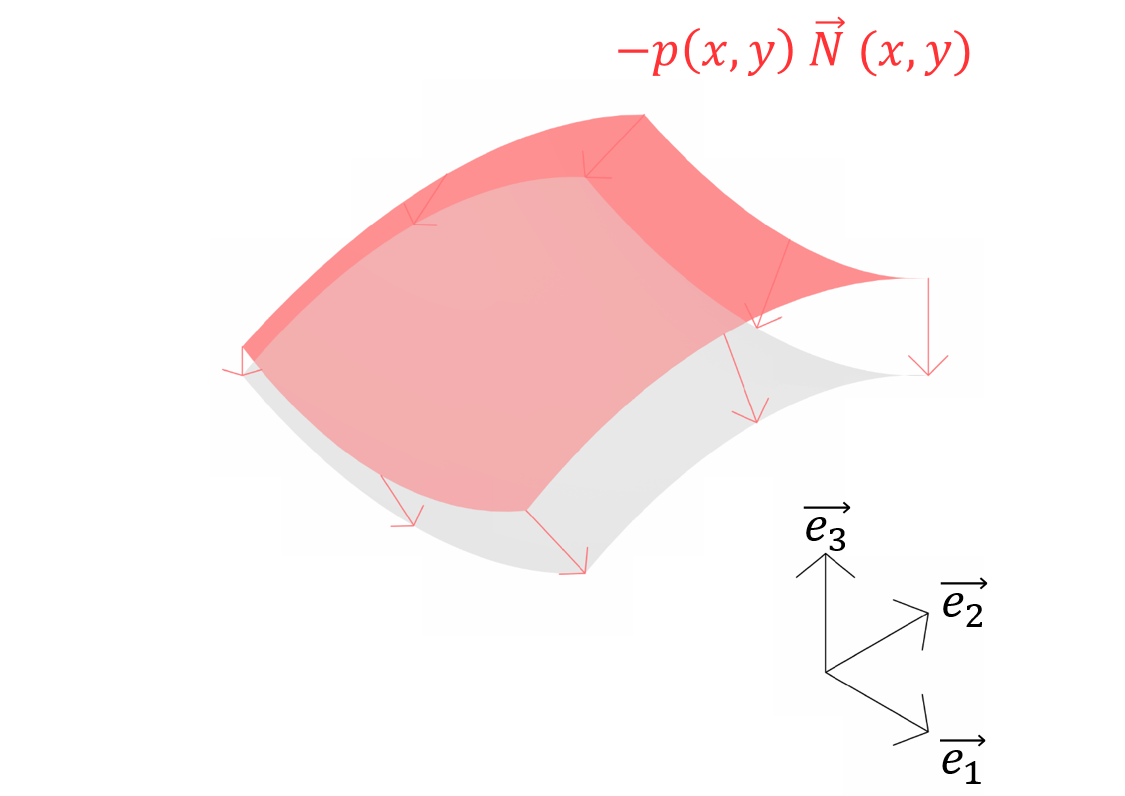}
     \caption{Surface element loaded with a pressure $p$}
     \label{Fig:pressure}
     \end{subfigure}
     \hfill
\begin{subfigure}{0.45\textwidth}
    \includegraphics[width =\linewidth]{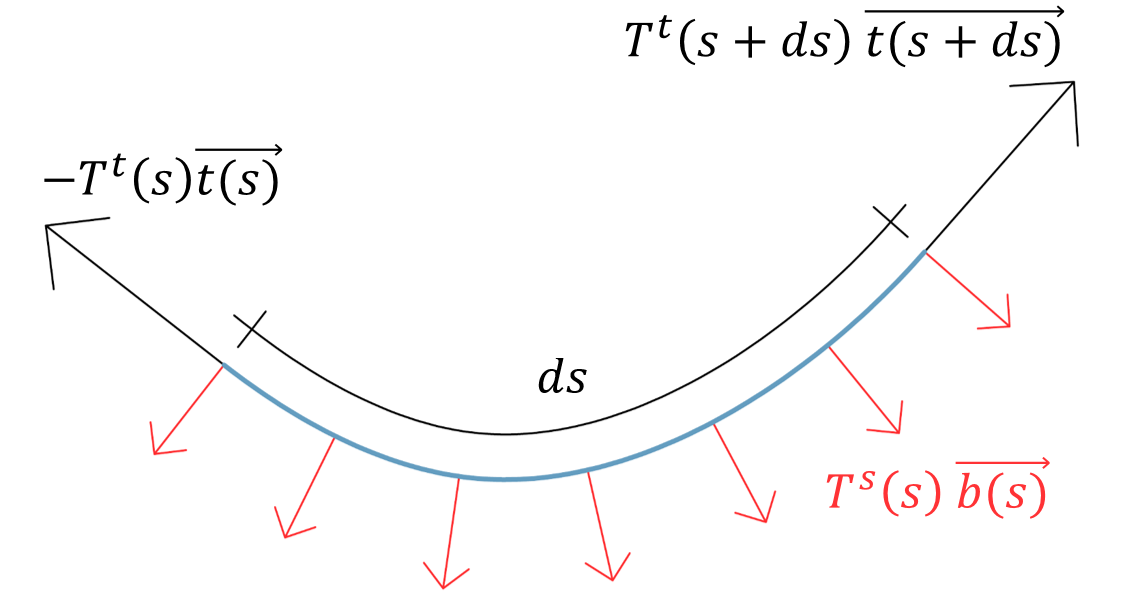}
    \caption{Equilibrium of a portion of thread.}
    \label{fig:equilibre_fil}
\end{subfigure}
     \caption{Equilibrium of minimal surfaces}
\end{figure}

Without loading ($p=0$) and for an isotropic stress $T_{11} = T_{22} = T$ and $T_{12}=0$, the Young-Laplace equation (\cref{eq:eq3}) becomes $e+g=0$ and the surface is minimal. Note that in reality, architectural structures are subject to out-of-plane loads, but pre-tension loads are the dominant loads for this type of structure.

\paragraph{\textbf{Equilibrium of a thread boundary}}
\label{sec:equi_thread}
\Cref{fig:munich,fig:ILEK} display cable boundaries: the boundary is a piecewise flexible wire with a fixed length. This problem is related to the thread problem \cite{dierkes_minimal_1992-2}.

A thread portion as illustrated in~\cref{fig:equilibre_fil} is considered: it is loaded at its ends by the rest of the thread and in the normal direction by the soap film. The tension $T^t$ in the thread is carried by the tangent to the wire. The surface tension $T^S$ is contained in the tangent plane to the surface, and the surface can slide along the thread; it therefore applies no tangential force along the thread. The thread is considered to have no thickness and no weight: it is thus not subject to gravity load. It is also totally flexible and inextensible. For a thread in equilibrium, all forces sum up to zero and the following system is obtained:
\begin{align}
\frac{dT^t}{ds} & = 0 \label{eq:uniform}  \longrightarrow \text{the tension is uniform} \\ 
T^t(s) \kappa_n(s) & = 0 \label{eq:asymp} \longrightarrow \text{the thread is an asymptotic line} \\ 
T^t(s) \kappa_g(s) + T^S(s) & = 0  \label{eq:k_cste} \longrightarrow \kappa_g \text{ is constant, because the film tension is constant}
\end{align}

For the same soap film (same soap, same film thickness, and thus $T^S$ is constant), if one pulls the wire, this changes its curvature inversely proportionally (\cref{eq:k_cste}): shape and forces are mutually dependent.\newline

Cables were used to build the IL building (see \cref{fig:ILEKaerial}) and the roof of the Munich Olympic Stadium (\cref{fig:munich}). Lines of curvature appear as good candidates for the location of cables on minimal surfaces: they are principal stress lines for prestress and for a uniform normal load (in the direction of principal stress, there is no shear but only normal stress).

The role of engineers and architects is to find a compromise between fidelity to the intended form, ease of fabrication and construction, and cost minimisation. These constraints can be translated into geometric constraints, as studied in the architectural geometry community. In this context, a new construction system was recently proposed by Eike Schling \cite{schling_design_2018}: building minimal surfaces (and more generally surfaces of negative Gaussian curvature) by placing beams along the asymptotic lines. From a construction standpoint, this makes it possible to build the beams from straight lamellas, which intersect at 90 degrees on minimal surfaces—thus simplifying the design of the node and preventing skew panels. Structurally, when a beam is applied vertically on a surface and follows an asymptotic line, it is twisted and bent in the direction $\overrightarrow{n_a}$ (see \cref{fig:strong_axis}), tangent to the surface. When it is placed horizontally and along a geodesic line, it is twisted and bent in the direction~$\overrightarrow{n_g}$, which is normal to the surface. Networks of asymptotic and geodesic curves can thus be constructed from bent straight lamellas, which explains their relevance in architectural geometry. Finally, as tensioned cables bordering the minimal surfaces are asymptotic lines, discretising the surface with asymptotic lines allows the grid to be aligned with the edge and the edge to be integrated with the structural frame. This construction system was developed and used in several pavilions \cite{eikeschling_asymptotic_2016,schling_design_2018,schling_geometry-based_2022}. A model and a structure are shown in \cref{fig:schling} and \cref{fig:insideout}. In the first case, the boundary is also an asymptotic line and is well integrated into the grid; in the other, the boundary is not an asymptotic line and intersects the grid at varying angles.

\begin{figure}[htb!]
\centering
\begin{subfigure}[t]{0.48\textwidth}
    \includegraphics[height = 5cm]{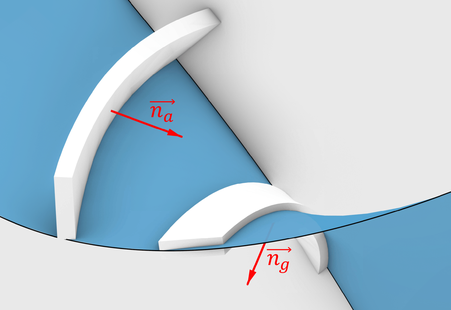}
    \caption{A beam placed on an asymptotic line and normally to the surface is bent in the tangent direction (left), a beam placed on a geodesic line and tangentially to the surface is bent in the normal direction.}
    \label{fig:strong_axis}
\end{subfigure}
\hfill
\begin{subfigure}[t]{0.48\textwidth}
    \centering
    \includegraphics[height = 5cm]{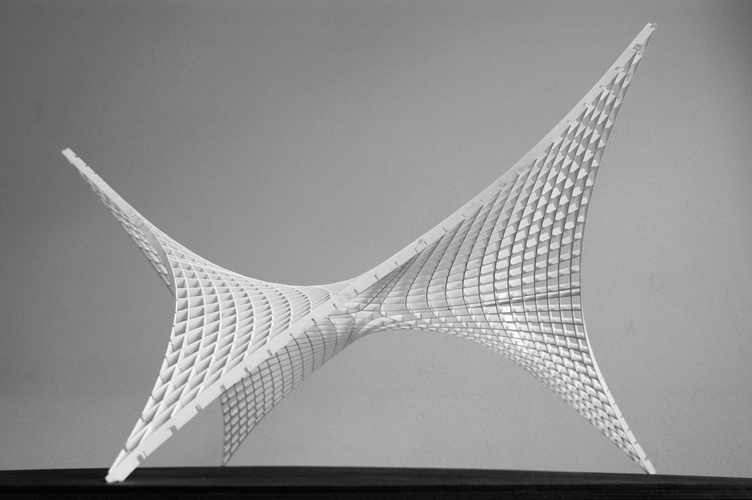}
    \caption{Asymptotic gridshell model, source \cite{eikeschling_asymptotic_2016}. This is not a minimal surface.}
    \label{fig:schling}
\end{subfigure}
\caption{Building with beams on asymptotic and geodesic curves.}
\label{fig:asymp_surface}
\end{figure}

\begin{figure}[htb!]
\centering
\begin{subfigure}[t]{0.48\textwidth}
    \centering
    \includegraphics[height = 5cm]{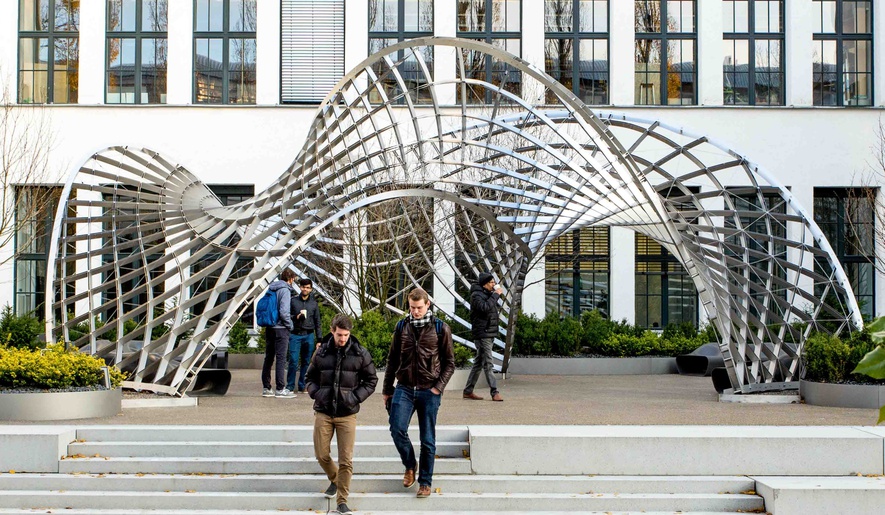}
    \caption{The gridshell is a minimal surface discretised along asymptotic lines.}
    \label{fig:insideout}
\end{subfigure}
\hfill\begin{subfigure}[t]{0.48\textwidth}
    \centering
    \includegraphics[height = 5cm]{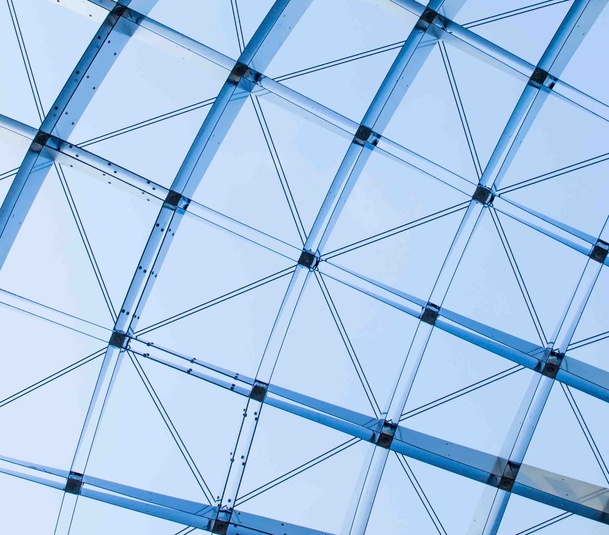}
    \caption{Beams intersect at right angles, and faces are braced by cables.}
    \label{fig:insideout_node}
\end{subfigure}
\caption{Inside\textbackslash Out Gridshell by Eike Schling \cite{schling_design_2018}, (pictures: Felix Noe).}
\label{fig:schling_insideout}
\end{figure}

The interest in minimal surfaces bounded by threads is therefore twofold in architecture: it corresponds to the natural geometry taken by cables, and in a beam system, the boundary—as an asymptotic line—can be constructed from a straight strip, well integrated into the network.\newline

In this paper, we will prove the existence of two families of surfaces bounded by asymptotic lines: tensile catenoids (minimal annuli bounded by a finite number of asymptotic lines of constant curvature), and tensile $k$-noids (minimal surfaces with the topology of the $k$-noids, each boundary component of which consists of four asymptotic lines of constant curvature joined by cusps). We will prove the following theorem:

\begin{theorem}\label{th:cat}
  For any $m\geq 3$, there exists a one-parameter family (after identifying by isometries of $\mathbb{R}^3$) of embedded minimal annuli, called {\rm tensile catenoids}. Any boundary curve of one such surface is the union of $m$ arcs, any one of them an asymptotic line of the surface with constant curvature. Two consecutive boundary arcs meet, forming a cusp (\textit{i.e.} they meet tangentlially, forming a corner on the surface of intrinsic angle zero). Moreover, any annulus in this family is symmetric with respect to a horizontal plane that divides the surface into two vertical graphs; and it also has $m$ vertical planes of symmetry that meet at an angle $\frac{\pi}{m}$.

  The parameter of the family is the length $\rho$ of the neck of the tensile catenoids (\textit{i.e.} the length of the intersection of each annulus with its horizontal plane of symmetry), going from zero (the limit when $\rho$ goes to zero is a piece of the horizontal plane of symmetry) to $2m\cos\frac{\pi}{m}$ (in this case, the annulus splits into $m$ disks, two consecutive ones joined by a point in their boundary).
\end{theorem}

Catenoids are relevant in architecture (for example, for cooling towers) and have already been studied at the IL (see \cref{fig:catenoidsoapIL,fig:catenoidNet}). We show in \cref{fig:catenoidBeams,fig:knoid_beams} how they could be constructed using beams, in the manner of Eike Schling.  The surface is discretised with beams along asymptotic lines: the rectangular beams are positioned on the surface so that their longer side is oriented along the normal direction, and are thus twisted. Nodes are orthogonal and flat. Quadrilaterals are regular, but their size varies, especially close to the boundaries where lines are concentrated.

\begin{figure}[htb!]
\centering
\begin{subfigure}[t][][b]{0.3\textwidth}
    \includegraphics[height = 3.5cm]{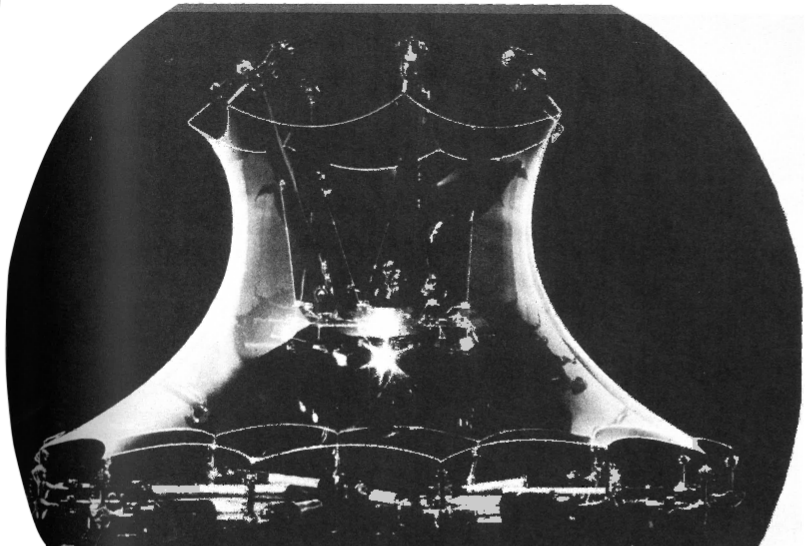}
    \caption{Soap model of a tensile catenoid. Source \cite{IL18}.}
    \label{fig:catenoidsoapIL}
\end{subfigure}
\hfill
\begin{subfigure}[t][][b]{0.25\textwidth}
    \includegraphics[height = 3.5cm]{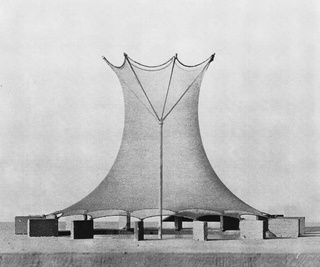}
    \caption{Model of tensile catenoid with a net. Source \cite{IL18}.}
    \label{fig:catenoidNet}
\end{subfigure}
\hfill
\begin{subfigure}[t][][b]{0.35\textwidth}
    \includegraphics[width = 3.5cm, rotate = -90, trim = {2.5cm 3cm 2cm 3cm}, clip]{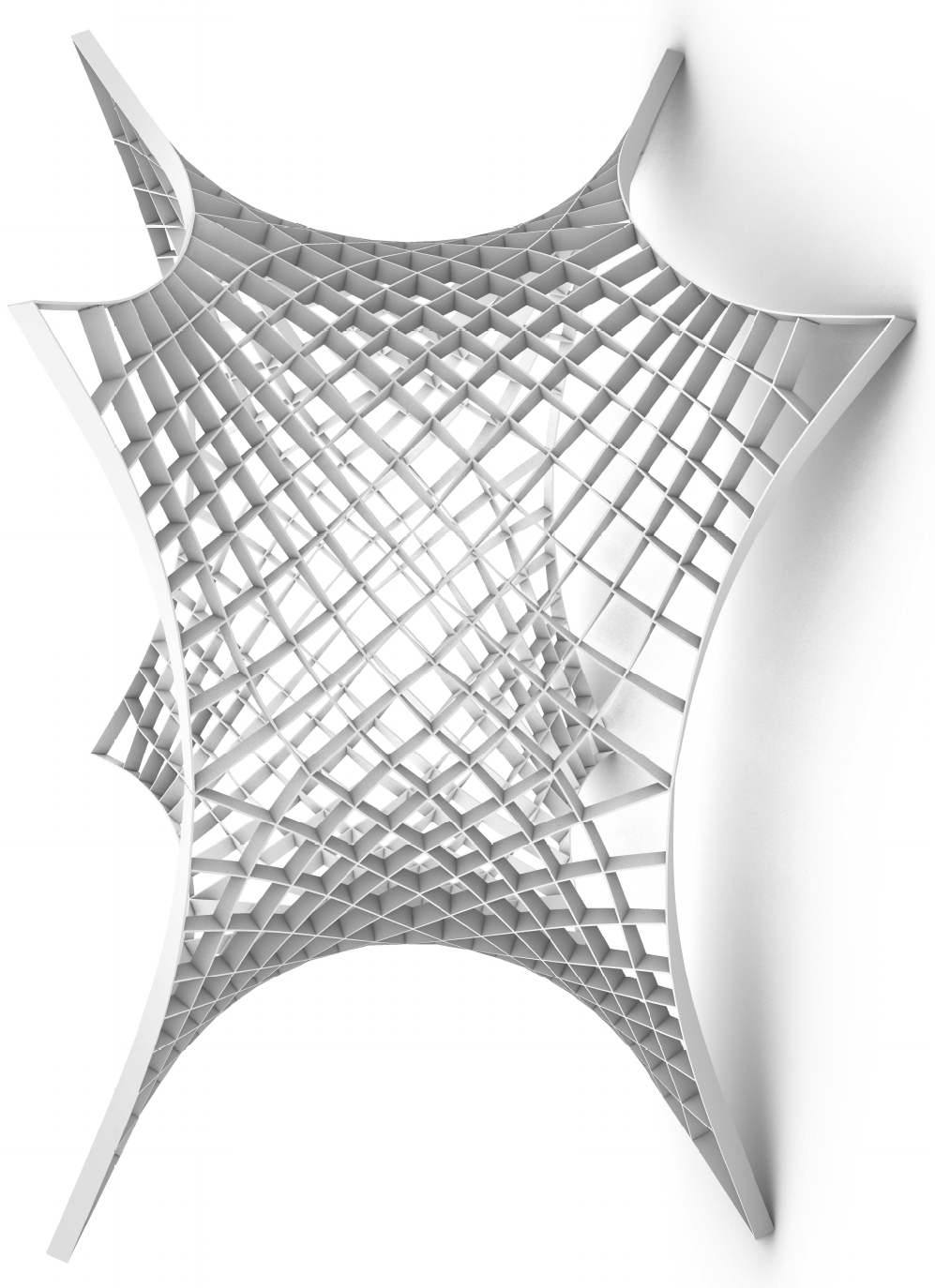}
    \caption{Model of a tensile catenoid made of beams following asymptotic lines (made using Rhino-Grasshopper).}
    \label{fig:catenoidBeams}
\end{subfigure}
\caption{Models of tensile catenoids with thread boundaries.}
\label{fig:catenoid}
\end{figure}

Other topologies can be obtained with the method described in this paper, such as tensile $k$-noids. In particular $3$-noids are constructed (\cref{fig:knoid_frise}). \cref{fig:catenoidBeams} shows a structure obtained from half a $3$-noid which is discretized with asymptotic curves. We will prove the following theorem:

\begin{theorem}\label{th:k}
  For any $k\geq 3$, there exists an embedded minimal surface with the topology of a sphere minus $k$ disks, called a {\rm tensile $k$-noid}. Any boundary curve of this surface is the union of four arcs, any one of them an asymptotic line of the surface with constant curvature, two consecutive ones meet forming a cusp (\textit{i.e.} they form a corner on the surface of angle zero). Moreover, the tensile $k$-noid is symmetric with respect to a horizontal plane dividing the surface into two vertical graphs; and it also has $k$ vertical planes of symmetry meeting at an angle $\frac{\pi}{k}$.
\end{theorem}

\begin{figure}
\centering
\begin{subfigure}[t]{0.45\textwidth}
    \includegraphics[height = 3.5cm]{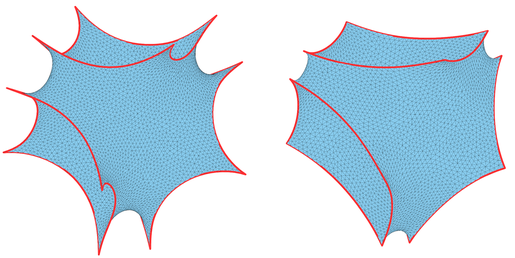}
    \caption{Computation of tensile $3$-noids with thread boundaries for two thread lengths.}
    \label{fig:knoid_frise}
\end{subfigure}
\hfill
\begin{subfigure}[t]{0.5\textwidth}
    \includegraphics[height = 4.5cm]{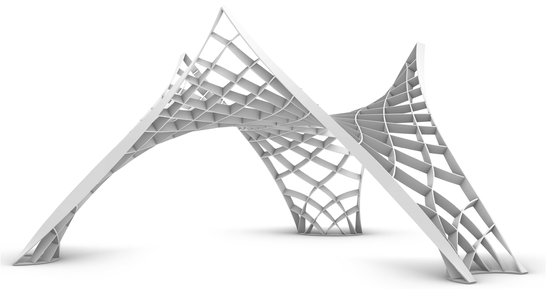}
    \caption{Model of half a tensile $3$-noid made of beams following asymptotic lines (made using Rhino-Grasshopper).}
    \label{fig:knoid_beams}
\end{subfigure}
\caption{Meshes and architectural interpretation of a 3-noid}
\label{fig:knoid_app}
\end{figure}

This article proposes proofs of the existence and embedding of two families of minimal surfaces bounded by asymptotic lines. It uses the
conjugate surface of the solution to a partially-free boundary problem.

\section{Preliminaries}

\subsection{The thread problem}
Minimal surfaces bounded by asymptotic lines with constant negative geodesic curvature are solutions to the thread variational problem. An introduction can be found in \cite{dierkes_minimal_1992-2}, chapter 10 and \cite{dierkes_regularity_2010}, chapter 5. For a collection of curves $\Gamma_1,\dots,\Gamma_k$ with endpoints $P^\ell_1$ and $P^\ell_2$ with $\ell \in \{1,\dots, k\}$, we consider movable curves $L_1,\dots,L_k$ with same endpoints but with fixed length $|L_\ell| > |P^\ell_2-P^\ell_1|$, in such a way that $<\Gamma_1,\dots,\Gamma_k, L_1, \dots,L_k>$ define a collection of closed curves $C_1,\dots,C_n$. The thread variational problem finds the minimising area surface bounded by $C_1,\dots,C_n$ with movable curves $L_\ell$ with fixed length  $|L_\ell|$. Alt \cite{alt_existenz_1973} and Ecker \cite{ecker_area-minimizing_1989} have proved the existence of the solution in higher dimensions.

\begin{theorem}[\cite{dierkes_regularity_2010}]
Every closed rectifiable curve $<\Gamma_1,\dots, \Gamma_k,L_1,\dots,L_k> \subset \mathbb{R}^3$ with movable boundary $L_1,\dots,L_k$ and fixed length $|L_\ell|> |P^\ell_2-P^\ell_1|$ spans a minimal surface $M$.
\end{theorem}

However, there is no control over the topology of the solution (even in cases where the physical experience has some evidence), which can be multi-connected or even not connected. There is no mathematical proof of the embeddedness of the movable boundary. Yet, Alt \cite{alt_existenz_1973} proved that the movable boundary is regular, has constant curvature and eventually meets the fixed boundary $\Gamma_1,\dots,\Gamma_k$ at some point in cuspidal singularity.

\medskip

For constructing solutions to the variational problem, we will use the Plateau-conjugate method (Karcher popularized this method in the 1980s, see ~\cite{karcher_construction_1989}).
In \cite{dierkes_regularity_2010} (p. 467), it is proved that a solution to this problem has a conjugate minimal surface satisfying a partially-free boundary condition with free boundary condition along movable boundary curves contained in spheres.

\subsection{Conjugate minimal surfaces}
It is well-known that a characterization for a surface to be minimal
is that its coordinate functions are harmonic. Two minimal surfaces
are called conjugate if their coordinate functions are conjugate as
harmonic maps. Hence, given a minimal surface, its conjugate surface
is locally well defined (globally if the surface is simply connected)
and it is unique up to a translation. Detailed information about
conjugate surfaces and the conjugate construction method for minimal
surfaces can be found in~\cite{karcher_construction_1989}. We include
some results that will be used later in order to fix the notation.

Given a (simply connected) minimal surface $M$, possibly with
boundary, we will denote by $M^*$ its conjugate surface. The surfaces
$M$ and $M^*$ are isometric, share the same Gauss map (\textit{i.e.} $N^*=N$)
and their shape operators differ by a rotation by angle $\pi/2$
(\textit{i.e.} $S^*=J\circ S$, with $J$ the $\pi/2$ rotation in the tangent
plane). In particular, given a curve $c$ in $M$, the conormal (resp. tangent) vector at a point $p$ on $c$ corresponds under conjugation (up to a sign) to the tangent (resp. conormal) vector at its conjugate point $p^*$ on the conjugate curve $c^*$ in $M^**$.

Given a curve i$c$ n $M$ (possibly in its boundary), if we denote by
$\kappa_g$ its geodesic curvature, $\kappa_n$ its normal curvature and
$\tau_g$ its geodesic torsion, then the corresponding data for its
conjugate curve $c^*$ in $M^*$ is given, respectively, by
\begin{equation}\label{eq:conj}
  \kappa_g^*=\kappa_g , \quad \kappa_n^* = - \tau_g, \quad \tau_g^* =
  \kappa_n .
\end{equation}

In our construction we will consider the conjugate surface $M^*$
of a minimal disk $M$ (bounded and simply connected) lying outside the
unit ball $\mathbb{B}$ such that $M$ is free-boundary to
$\mathbb{S}^2=\partial \mathbb{B}$ along 
$\ell=\partial M\cap \mathbb{S}^2$ (\textit{i.e.} the exterior conormal vector of $M$ at
any point of $\ell$ will point to the origin).

The boundary of $M$ will consist of the union of a finite number of
smooth curves. More precisely, $\partial M-\ell$ will be the union of
a finite number of straight segments $\ell_i$. The corresponding
conjugate curve $\ell_i^*\subset \partial M^*$ of each $\ell_i$ will
be a geodesic contained in a plane orthogonal to $\ell_i$, and $M^*$
could be extended by symmetry with respect to that plane.

The next subsection includes a result related to partially-free boundary
problems that justifies the existence of these kinds of minimal surfaces
just described.

\subsection{Partially-free boundary problem}

Let $S\subset\mathbb{R}^3$ be a complete surface (in our case, it will be $S=\mathbb{S}^2$) and $\Gamma\subset\mathbb{R}^3$ be a piecewise smooth regular Jordan curve with endpoints $p,q\in S$ such that $\Gamma-\{p,q\}$ is disjoint from $S$ and $p,q$ can be connected with an arc on $S$. In~\cite{dierkes_minimal_1992} (chapters~4 and~5), they consider the partially-free (or semi-free) boundary problem of looking for a minimal disk with boundary $\Gamma\cup\ell$, where $\ell$ is contained in $S$. In fact, they consider a more general setting, but we only focus on this case as it is simpler and the required one below. They also study hypotheses under which the obtained minimal disk meets the surface $S$ orthogonally (\textit{i.e.} when the solution to the partially-free boundary problem is stationary within the configuration $\langle\Gamma, S\rangle$).

\begin{theorem}[\cite{dierkes_minimal_1992}]\label{th:semifree}
  Let $S\subset\mathbb{R}^3$ be a complete surface of class $C^1$ and $\Gamma\subset\mathbb{R}^3$ be a piecewise smooth regular Jordan curve with endpoints $p,q\in S$ such that $\Gamma-\{p,q\}$ is disjoint from $S$ and $p,q$ can be connected with an arc on $S$. Then there exists a minimal disk contained in $\R^3-S$ with boundary $\Gamma\cup\ell$, where $\ell$ is contained in $S$, that meets the surface $S$ orthogonally along $\ell$; that is, there exists a stationary solution to the partially-free boundary problem within the configuration $\langle\Gamma, S\rangle$. Moreover, $M$ is a minimum for the energy (and so the area) among all the surfaces with these properties (see \cite{dierkes_minimal_1992} p.331).
\end{theorem}

Now we consider $S=\mathbb{S}^2$ in the theorem above, and let us denote by $\Sigma$ the obtained minimal disk and $\ell$ is the free part of its boundary.

\begin{claim}\label{cl:curvatureline}
If $S=\mathbb{S}^2$ in Theorem \ref{th:semifree}, the curve $\ell$ is
  a line of curvature of $M$ (hence its geodesic torsion $\tau_g$
  vanishes identically) with geodesic curvature $\kappa_g=1$ at any point
  (up to a possible change of orientation).
\end{claim}

Since $M$ is orthogonal to $\s^2$ along $\ell$, we obtain using
Joachimsthal's Theorem that $\ell$ is a line of curvature of
$M$. Moreover, in absolute value, the geodesic curvature vector of
$\ell$ in $M$ coincides with the normal curvature of $\ell$ in
$\mathbb{S}^2$. Hence $|\kappa_g|=1$
  at any point. This proves the claim.

Identities~\eqref{eq:conj}
say in this particular case that, along the conjugate curve
$\ell^*\subset M^*$, we have
\begin{equation}\label{eq:relacion}
\kappa_g^*=\kappa_g=1 , \quad \kappa_n^* = - \tau_g=0, \quad \tau_g^* =
  \kappa_n.
\end{equation}
We then obtain that $\ell^*$ is an asymptotic curve of $M^*$ with
constant (geodesic) curvature $1$.

\section{Construction method for solutions to some thread problems}

In this section, we are going to construct, using the
Plateau-conjugate method from the solution of a partially-free
boundary problem, a surface solution to the thread problem explained
above.

Let $\Gamma\subset\mathbb{R}^3-\mathbb{B}$ be a piecewise smooth curve
with endpoints $\widehat {p_1}$, $\widehat {p_2}$ in $\mathbb{S}^2$. (In Subsections~\ref{sub:cat} and
\ref{sub:k} $\Gamma$ will be the union of finitely many straight segments.)
By Theorem~\ref{th:semifree}, there exists a minimal disk $M$ with
boundary $\Gamma\cup\ell$, with $\ell=\partial M\cap\mathbb{S}^2$, so
that $M$ and $\mathbb{S}^2$ intersect orthogonally along
$\ell$. ($\Gamma$ is a rigid wire whereas $\ell$ is a thread.)
As proved above (see Claim~\ref{cl:curvatureline}), $\ell$ is a line
of curvature of $M$, and the geodesic curvature of $\ell$ (as a curve
in $M$) is $1$ at any point. 
In this section, we will first obtain general properties of $M$ and $\ell$ under certain assumptions, and then we will construct particular cases that will allow us to obtain the desired surfaces described in the introduction.

\begin{remark}\label{r:transverse}
  Since the exterior conormal to $M$ along $\ell$ points to the origin
  of $\mathbb{R}^3$ at any point $p\in\ell$, the tangent plane
  $P$ of $M$ at $p$ passes through the origin; its direction is
  determined by~$\ell'$. By the maximum principle at the boundary we
  obtain that either $M$ is contained in $P$ or it has points on both
  sides of $P$ around $p$. 
\end{remark}

Let us now prove a kind of ``convex hull'' property for one such
surface $M$, illustrated in~\cref{fig:lemma4.3}.

\begin{lemma}\label{l:CHull}
  With the notation above, if $\Pi$ is a plane passing through the
  origin, $\Pi^+$ is a component of $\mathbb{R}^3-\Pi$ and $\Gamma\subset
  \overline{\Pi^+} = \Pi^+\cup \Pi$, then the interior of both $M$ and
  $\ell$ are contained in~$\Pi^+$.
\end{lemma}

\begin{figure}[htb!]
    \centering
    \includegraphics[width=0.5\linewidth]{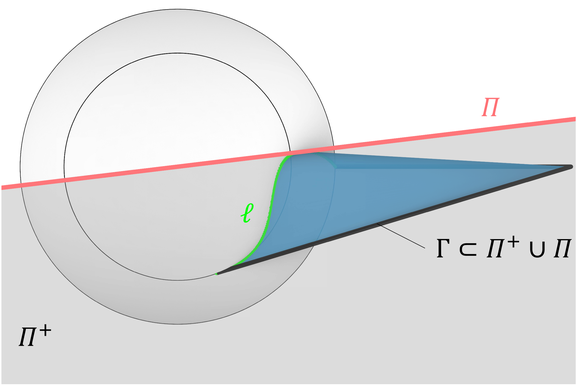}
    \caption{Projection of a configuration of
      Lemma~\ref{l:CHull}. }
    \label{fig:lemma4.3}
\end{figure}

\begin{proof}
  Let us suppose there exist points in $M\cap \Pi^-$, where
  $\Pi^-\cup \Pi^+=\mathbb{R}^3-\Pi$.  By the convex hull property, any
  connected component of $M\cap \Pi^-$ contains points of $\ell$ in its
  boundary. Consider the (non-smooth) surface $\tilde M$ obtained by
  replacing those components with the corresponding regions in
  $\Pi-\mathbb{B}$. (We observe that $\Pi-\mathbb{B}$ is free-boundary.)
  Thus $\tilde M$ is another solution to the partially-free boundary
  problem related to $\Gamma$ and $\mathbb{S}^2$ with less area
  than~$M$, a contradiction. Hence $M\subset \Pi\cup \Pi^+$.

  By the maximum principle, $M$ cannot have interior points in $\Pi$,
  so only points in $\partial M$ can touch the plane, and then the
  interior of $M$ is contained in $\Pi^+$. By the maximum principle at
  the boundary (using that both $M$ and $\Pi$ are tangent at any
  common point in $\ell$), we also get that no interior point of
  $\ell$ can be contained in $\Pi$.
  \end{proof}

 A similar proof using that $M$ is a minimum for the area functional can be applied to obtain the following result

 \begin{lemma}\label{l:P}
If  $\Pi$ is a plane passing through the origin and $M'\subset
   M$ is a component with boundary contained in $\mathbb{S}^2\cup
   \Pi$, then $M'$ (and so $M$) is contained in $\Pi$.
\end{lemma}

 We observe that, in the hypothesis of Lemma~\ref{l:CHull}, we get
 that the boundary curve $\ell$ lies in the hemisphere
 $\mathbb{S}^2\cap \Pi^+$. Since this hemisphere projects graphically on
 $\Pi$ and $\ell$ is embedded, we conclude that the orthogonal
 projection of $\ell$ on $\Pi$ is injective. The following proposition
 says that, under some extra hypothesis, the whole surface $M$
 projects graphically on the plane.

\begin{proposition}\label{p:graph}
  Given a plane $\Pi$ passing through the origin, let us suppose that
  the following properties are satisfied (with the notation above):
\begin{enumerate}
  \item[(a)] The planes orthogonal to $\Pi$ separate $\Gamma$ in at
  most three components.
\item[(b)] $M$ is contained in a convex wedge region $W$ bounded by
  two half-planes orthogonal to $\Pi$ whose common boundary $L_W$
  passes through the origin, each half-plane containing an endpoint of
  $\ell$. 
\item[(c)] The plane $P_0$ orthogonal to $\Pi$ passing through the two
  endpoints $\widehat {p_1}, \widehat {p_2}$ of $\ell$ separate $\Gamma$ from the origin possibly in a non strictly way i.e.
  $$\R^3 \setminus P_0=P_0^+ \cup P_0^- \hbox{ with } 0 \in \overline{P_0^-},\;\; \Gamma \subset \overline{P_0^+} \hbox{ but } \Gamma \not\subset P_0.$$
\item[(d)] Up to straight segments contained in $\Gamma$ and
  orthogonal to $\Pi$ (if any), $\partial M= \Gamma \cup \ell$ is
  projected injectively onto a Jordan curve in $\Pi$ which bounds a
  domain $\Omega$.
\end{enumerate}
Then $M$ is a graph over the domain $\Omega\subset\Pi$.
\end{proposition}

\begin{remark}\label{r:convex}
We observe that condition $(a)$ is satisfied when the projection of
$\Gamma$ over $\Pi$, together with the straight segment joining its
endpoints, is a convex domain.
\end{remark}

\begin{proof}
  First of all, we can assume (to simplify the notation) that
  $\Pi=\{z=0\}$. 
  
  By $(d)$, it suffices to prove that $M$ is a multigraph over
  $\Pi$. Let us then suppose by contradiction that there exists a
  point $p\in M$ with vertical tangent plane $P=T_p M$. Locally at
  $p$, $M\cap P$ consists of at least two curves that form an equiangular system, where $M$ passes from one side to another of the
  plane $P$.

  By the maximum principle, the curves in $M\cap P$ cannot form a
  closed loop, as we would obtain a bounded domain in $M$ whose
  boundary is entirely contained in $P$, and $M$ would be contained
  in~$P$. For the same reason, no two of these curves can have a
  common endpoint in $\partial M$, nor two of them can have endpoints
  contained in the same vertical segment in $\Gamma$, if any. In
  particular, the curves in $M \cap P$ disconnect $M$ into at least
  four connected components and $\partial M\cap P$ consists of at
  least four points.

  By hypothesis $(a)$, at most two points in $\Gamma$ can be endpoints
  of curves in $M \cap P$, so $P$ intersects $\ell$ in at least two
  points.  In particular, there exists a component $M'$ of $M-P$
  whose boundary is contained in $\ell\cup P$. By Lemma~\ref{l:P},
  if $P$ passes through the origin, we get that $M'$ is contained
  in~$P$, a contradiction. Thus $P$ cannot contain the origin. Call $P^-,P^+$ the connected components of $\R^3-P$ such that the origin is contained in $P^-$.

  From the fact that the curves in $M\cap P$ must have different endpoints in $\partial M$, we deduce that there are at least two connected components of $M-P$ in $P^+$, called $M_1^+$ and $M_2^+$.
  
  Hypothesis $(a)$ says that $\Gamma-P$ can have at most three components.   
  If it has one or two components, then we can assume that $\Gamma\cap P^+$ has at most one connected component contained (if any) in $\partial M_1^+$, and  $\partial M_2^+$ is contained in $\ell\cup P$.

  Let us now suppose that $\Gamma-P$ has three components, and call $\Gamma^+$ the component not having $\widehat {p_1}$ nor $\widehat {p_2}$ (defined in hypothesis $(c)$) as endpoint. We recall that $M\subset W$ by hypothesis $(b)$. Then hypothesis $(c)$ says in particular that $P_0$ separates $\Gamma\cap P$ from the origin. Thus $\widehat {p_1},\widehat {p_2}\in P^-$, and $\Gamma^+\subset P^+$. Assume that $\Gamma^+\subset\partial M_1^+$. Hence in this case we also get that $\partial M_2^+\subset \ell\cup P$.
   
By the
  maximum principle, the distance from $M_2^+$ to $P$ attains its
  maximum value at a point $q\in \ell$.
 On the other hand, since the inner conormal $\overrightarrow{n_q}$
  to $M$ along $\ell$ at $q$ coincides with its position vector, we
  get $\langle \overrightarrow{n_q}, N_P\rangle >0$,  where $N_P$ denotes the Gauss map of $P$ that
  points to $P^+$. In particular, there must be
  interior points of $M_ 2^+$ farther from $P$ than $q$, a
  contradiction.  This contradiction proves that $M$ is a multigraph
  on $\Pi$, and Proposition~\ref{p:graph} follows.
\end{proof}

\begin{figure}
    \centering
    \includegraphics[width=0.8\linewidth]{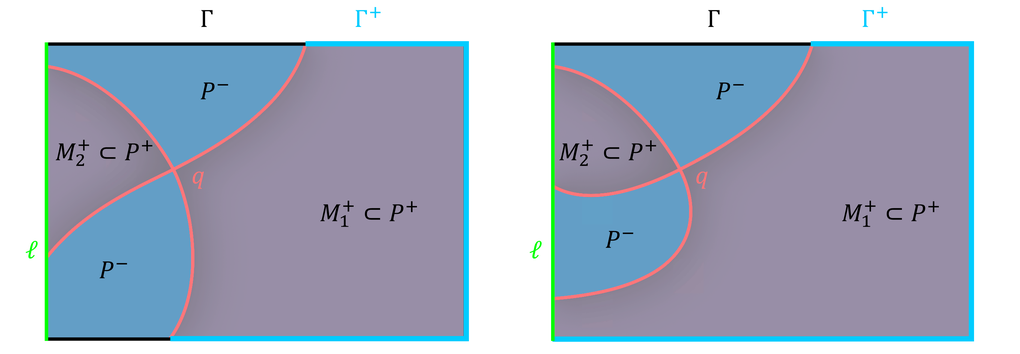}
    \caption{Domains with the preimages of the boundary (black for $\Gamma$ and green for $\ell$) and intersection curves
(in red), showing the subdomains that are mapped to $P^+$ and $P^-$. Left: case where $\Gamma-P$ consists of three components, right: case where $\Gamma-P$ consists of exactly two components.}
    \label{fig:domainP+P-}
\end{figure}

\begin{lemma}\label{l:radial}
In the hypothesis of Proposition~\ref{p:graph}, if $\partial M$ does
not intersect $L_W$, defined in $(b)$, then the projection of $\ell$ in
$\Pi$ is a radial graph from the origin.
\end{lemma}

\begin{proof}
Let $P$ be a plane such that it  contains $L_W$ (\textit{i.e.} it is orthogonal to $\Pi$ and it passes through the origin) and it intersects the wedge $W$.  
Since each endpoint of $\ell$ is contained in a different half-plane in the boundary of the wedge $W$, we get that $P$ separates the endpoints of $\ell$, and $\ell\cap P$ is not empty. 

Let us prove that $P$ can
intersect $\ell$ in exactly one point. Suppose this is not the case. 
After possibly rotating the plane slightly about $L_W$  to avoid tangency points, we can
suppose that $P$ intersects $\ell$ transversely in at least three points.
By the maximum principle, the curves in $M\cap P$ cannot form a closed loop nor two of these curves can have a common endpoint in $\partial M$ (nor in the same vertical segment in $\Gamma$, if any).
By hypothesis $(a)$, $P$ intersects $\Gamma$ in at most two points. Hence there must exist a component of $M - P$ whose boundary is
contained in $\ell\cup P$. We reach a contradiction
with Lemma~\ref{l:P}.
 \end{proof}

We next show that, in the setting above, $\ell$ is a convex
(or concave) curve on the sphere under some extra assumptions.

\begin{proposition}\label{p:kg}  
  If the interior of $\Gamma$ is contained strictly on one side of the
  plane $P_1$ passing through its endpoints and the origin, and furthermore
  any plane divides $\Gamma$ in at most three connected components, then the
  geodesic curvature of $\ell$ with respect to $\mathbb{S}^2$ never
  vanishes. 
\end{proposition}

\begin{proof}
  Let us suppose there exists a point $p\in \ell$ where the geodesic
  curvature $k^{\s^2}_g$ with respect to $\mathbb{S}^2$ vanishes; equivalently, the
  normal curvature of $\ell$ at $p$ in $M$ is zero. Since $\ell$ is a
  line of curvature of $M$, we obtain that $p$ is a flat point
  (\textit{i.e.}  $K(p)=0$). Additionally if at $p$, we have $\partial k^{\s^2}_g(p)=0$, then $\nabla K(p)=0$.  In particular, if we denote by $P$ the
  tangent plane of $M$ at $p$ then $P \cap \overline M$ consists
  locally around $p$ of at least three curves emanating from~$p$. 
  
  If at $p$, $\ell$ is transverse to $P \cap M$, then $P\cap M$ has at least three curves emanating from $p$ into the interior of M as
  illustrated in~\cref{fig:prop4.5}-left.  If at $p$, $\nabla K(p)\neq 0$ and $\ell$ is tangent to $P$, then $\ell$ is not at the same side of $P$ around $p$ (it has an inflection point) as
  illustrated in~\cref{fig:prop4.5}-center. If at $p$, $\nabla K(p)=0$, we have at least four curves emanating from~$p$ as
  illustrated in~\cref{fig:prop4.5}-right.

\begin{figure}[htb!]
    \centering
    \includegraphics[width=\linewidth]{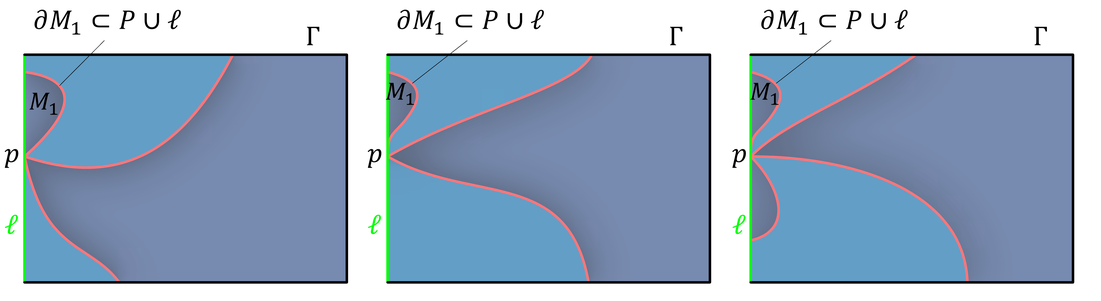}
    \caption{Schematic picture of the minimal surface $M$. Its
      boundary curves are drawn in black ($\Gamma$) and yellow
      ($\ell$). In red we have the intersection curves of $M$ and $P$. }
    \label{fig:prop4.5}
  \end{figure}
  
 By the maximum principle
  $\overline M\cap P$ cannot contain any closed loop, so the
  intersection curves can be extended to $\partial M$ and they do not
  intersect each other, even at boundary points. 

  Since $P$ divides $\Gamma$ in at most three components, in any case
  we conclude that $M-P$ has at least four connected components, at
  least one of them $M_1$ with boundary in $P\cup \ell$.  By
  Lemma~\ref{l:P} we get $M_1\subset P$, a contradiction.
\end{proof}

\subsection{Construction of the tensile catenoids}
\label{sub:cat}

In this section we prove the existence of the family of minimal annuli
bounded by a finite number of asymptotic curves of constant curvature
described in Theorem~\ref{th:cat}, called tensile catenoids (see
~\cref{fig:catenoid}).  They will be constructed by symmetry from a
fundamental piece obtained as the conjugate surface of the solution to
a partially-free boundary problem. Both conjugate surfaces are shown
in~\cref{fig:conj_correspondance}. We will start by describing
the partially-free boundary problem. Next, we will study the
properties of the resulted surface as a solution to that
problem. Finally, we will conjugate the surface and extend it by
symmetries.

\begin{figure}[htb!]
\centering
\begin{subfigure}[t]{0.5\textwidth}
\centering
    \includegraphics[height = 5cm]{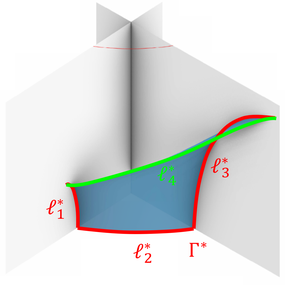}
    \label{fig:thread_ell}
\end{subfigure}
\hfill
\begin{subfigure}[t]{0.45\textwidth}
\centering
    \includegraphics[height = 4cm]{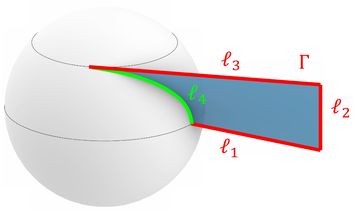}
    \label{fig:free_ell}
\end{subfigure}
\caption{The fundamental piece of a tensile catenoid and its conjugate
  minimal surface, solution of a partially-free boundary problem. The
  corresponding curves by conjugation are highlighted in the same
  colour: The union of planar lines of curvature $\Gamma^*$ (in red) is
  sent to the curve $\Gamma$ made of straight segments. The movable
  boundary curve $\ell_4^*$ (in green) corresponds to a free-boundary
  curve $\ell_4$ on the sphere.}
\label{fig:conj_correspondance}
\end{figure}

Let us consider the points $p_1=(0,1,0)$, $p_2=(0,1+a,0)$ and
$p_3=(0,1+a,h)$, for any $a>0$ and $h>0$ to be determined. We call
$\ell_1$ (resp. $\ell_2$) the straight segment joining $p_1$ and $p_2$
(resp. $p_2$ and $p_3$). We define $\ell_3$ as the horizontal straight
segment going from~$p_3$ to the point
$p_4\in\mathbb{S}^2\cap\{z=h,x>0\}$ so that $\ell_3$ meets the circle
$\mathbb{S}^2\cap\{z=h\}$ tangentially, see
~\cref{fig:configuration}.

\begin{figure}[htb!]
  \centering 
  \includegraphics[width=.6\textwidth]{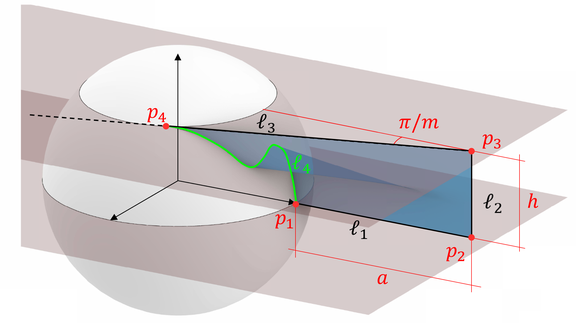}
  \caption{The piecewise smooth curve
    $\Gamma=\ell_1\cup \ell_2\cup \ell_3$ together with a curve
    $\ell_4$ in $\mathbb{S}^2$ joining the endpoints of $\Gamma$ and a surface
    bounded by $\Gamma\cup\ell_4$ meeting the sphere
    orthogonally.}
  \label{fig:configuration}
\end{figure}

We will also need to close the conjugate surface after reflections, so
we require the angle between the vertical planes containing $\ell_1$
and $\ell_3$, respectively, to be of the form $\frac{\pi}{m}$, for
some prescribed $m\geq 3$. We observe that once we fix~$m$, the
parameters $a, h$ depend on each other, so we have a 1-parameter
family of configurations. The parameter in our construction ($m$ is
fixed) can be considered to be $h\in(0,\cos\frac\pi m)$; thus
$a=\sqrt{1-h^2} \csc\frac\pi m -1$ (
See~\cref{fig:param}).

\begin{figure}[htb!]
  \centering \includegraphics[width=.8\textwidth]{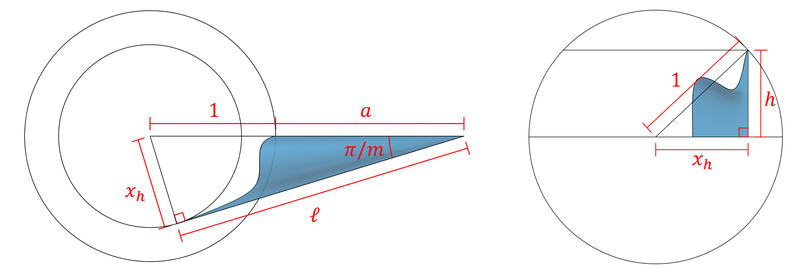}
  \caption{As $\ell_1$ is
    perpendicular to $\mathbb{S}^2$ and $\ell_3$ is tangent to
    $\mathbb{S}^2$, the parameters $h$ and $a$ depend on each other once we fix the angle $\frac\pi m$.} 
  \label{fig:param}
\end{figure}

By Theorem \ref{th:semifree}, there exists an embedded minimal surface
$M$ outside the unit ball $\mathbb{B}$ with
$\Gamma=\ell_1\cup\ell_2\cup\ell_3\subset\partial M$ and so that $M$
meets~$\mathbb{S}^2$ orthogonally along the smooth embedded curve
$\ell_4=\partial M-\Gamma\subset\mathbb{S}^2$. Moreover, $\ell_4$ is a
line of curvature of $M$ with constant geodesic curvature one (see
Claim~\ref{cl:curvatureline}).  We are going to prove that this
solution to the partially-free boundary problem is a graph in the
direction of the straight segments $\ell_1, \ell_2, \ell_3$ and
it is unique. First, let us describe the region containing the
surface~$M$.

\begin{lemma}\label{l:ConvexHull}
  Let us call $P_{p_4,z}$ the vertical plane passing through the
  origin and $p_4$, and $W_z$ the (open) wedge
  region from $P_{p_4,z}$ to $\{x=0\}$ with interior angle
  $\frac{m-2}{2m}\pi$ (it only contains points with
  positive $x$ and $y$ coordinates). Then the interior of both $M$ and $\ell_4$ are
  contained in $W_z\cap\{0<z<h\}$. Moreover:
\begin{enumerate}    
\item $M$ lies below the plane $P_{O,p_3,p_4}$ passing through
  $p_3,p_4$ (\textit{i.e.} containing $\ell_3$) and the origin.
\item $M$ lies above the plane $P_{O, p_1,p_4}$ passing through
  $p_1,p_4$ and the origin.
\item $M$ lies on one side of the vertical plane $P_{\ell_3,z}$
  containing $\ell_3$.
\end{enumerate}
\end{lemma}

Most of the planes mentioned in Lemma \ref{l:ConvexHull} are shown in~\cref{fig:planes}.

\begin{figure}[htb!]
\centering 
\includegraphics[width=\textwidth]{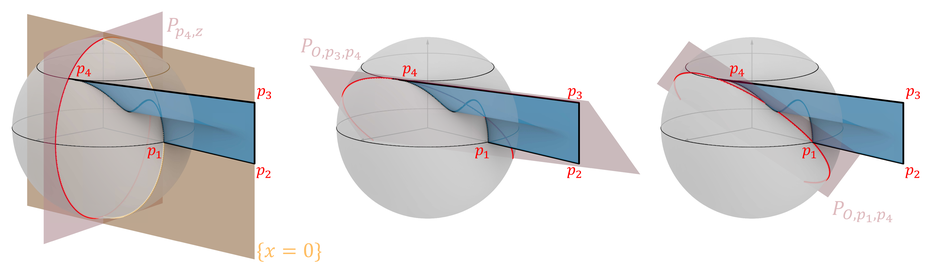} 
\caption{Perspective view of planes delimiting the surface. }
\label{fig:planes}
\end{figure}

\begin{proof}
  First, we observe that $\Gamma$ is contained in the closure of
  $W_z\cap\{z>0\}$. Then Lemma~\ref{l:CHull} applied to $\{z=0\}$,
   $\{x=0\}$ and $P_{p_4,z}$, ensures that the interior of both
  $M$ and $\ell_4$ are contained in $W_z\cap\{z>0\}$.

  Let us now call $\gamma$ the geodesic of $\mathbb{S}^2$ passing
  through $p_4$ which is tangent to the circle $\mathbb{S}^2\cap
  \{z=h\}$. Let us denote by $P_\gamma$ the plane containing
  $\gamma$. As $\gamma$ is symmetric with respect to $P_{p_4,z}$, then
  $\gamma\cap\{z=0\}$ is contained in $P_{p_4,z}^{\perp}$, the
  vertical plane orthogonal to $P_{p_4,z}$ and passing through the
  origin. The plane $P_{p_4,z}^{\perp}$ leaves $\Gamma$ on one
  side. By rotating this plane around the horizontal line connecting
  the two points of $\gamma\cap\{z=0\}$ until arriving at $P_\gamma$
  we observe that $P_\gamma$ leaves $\Gamma$ on one side (this
  operation is illustrated in~\cref{fig:lemme4_8_rotate1}). By
  Lemma~\ref{l:CHull}, we get that both $\ell_4$ and $M$ lie on the
  same side of $P_\gamma$ as~$\Gamma$ (\textit{i.e.} below). In
  particular, $\ell_4$ lies below $ \{z=h\}$. Thus $\partial
  M\subset\{z\leq h\}$, and $M$ is contained in $W_z\cap\{0<z<h\}$.
  
\begin{figure}[htb!]
    \centering
    \includegraphics[width=0.75\linewidth]{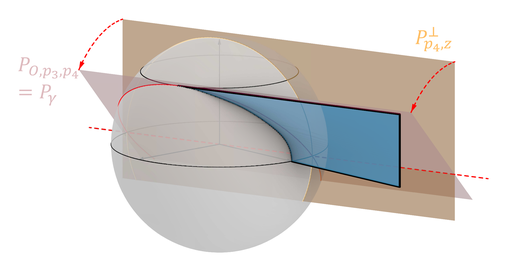}
    \caption{Rotation of the plane $P_{p_4,z}^{\perp}$ around the axis
      connecting the points of $\gamma \cap \{z=0\}$ (dashed red line)
      to prove that the minimal surface lies below the plane
      $P_{\gamma}$. }
    \label{fig:lemme4_8_rotate1}
\end{figure}

Since $\ell_3\subset$ $P_{O,p_3,p_4}\cap\{z=h\}$ is tangent to the
sphere, we get that $P_{O,p_3,p_4}\cup\mathbb{S}^2$ is nothing but the great
circle $\gamma$. Thus $P_{O,p_3,p_4}$ = $P_\gamma$, and the first item
has already been proven.

  If we rotate the plane $\{z=0\}$ about the $y$-axis until we reach
  the plane $P_{O,p_1,p_4}$, we observe that $\Gamma$ lies above that
  plane (this operation is illustrated in
 ~\cref{fig:lemme4_8_rotate2}). We then conclude the second item
  by using Lemma~\ref{l:CHull} in this situation.

  \begin{figure}[htb!]
    \centering
    \includegraphics[width=0.75\linewidth]{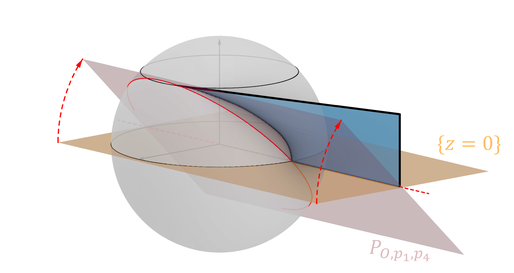}
    \caption{Rotation of the plane $\{z=0\}$ around the $y-$axis
      (dashed red line) to prove that the minimal surface lies above
      the plane $P_{O,p_1,p_4}$.}
    \label{fig:lemme4_8_rotate2}
\end{figure}

Finally, we observe that $\Gamma$ lies on one side of the vertical
plane $P_{\ell_3,z}$ (but this plane does not contain the origin). Let
us suppose there exist points of $M$ on the other side of
$P_{\ell_3,z}$. By the maximum principle, the farthest point $p$ in
$\overline M$ from $P_{\ell_3,z}$ on this side must be contained in
$\ell_4$. Denote by $\overrightarrow{n}$ the normal vector to
$P_{\ell_3,z}$ pointing to the side containing $p$, see
~\cref{fig:lemme4_8}. As the interior conormal
$\overrightarrow{n_p}$ of $M$ at $p$ is pointing outside the sphere
and $M\subset W_z$, we get that $\langle
\overrightarrow{n_p},\overrightarrow{n} \rangle > 0$ and then there
must be points in the interior of $M$ that are farther from
$P_{\ell_3,z}$ than $p$, a contradiction (see
~\cref{fig:lemme4_8}). This concludes the proof of
Lemma~\ref{l:ConvexHull}.
  \end{proof}

  \begin{figure}[htb!]
      \centering
      \includegraphics[width=0.7\linewidth]{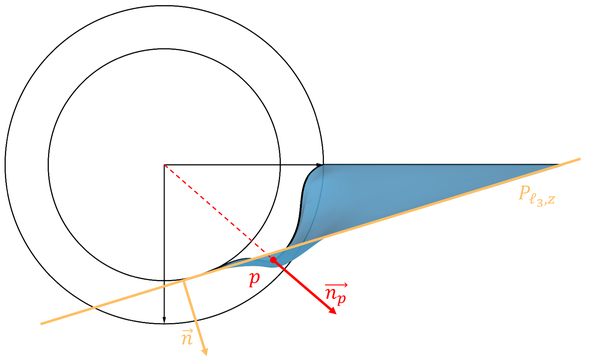}
      \caption{If there are points of $M$ on the other side of the vertical plane $P_{\ell_3,z}$, these points must be on the boundary. But as the surface is orthogonal to the sphere, there are interior points farther away from $P_{\ell_3,z}$, which is a contradiction.}
      \label{fig:lemme4_8}
  \end{figure}

  \begin{remark}\label{r:cusp}
   At any point of $\ell_4$, the position vector and $\ell_4'$
   (tangent to the sphere) form an orthogonal basis of the tangent
   plane of $M$. By continuity the same happens at $p_4$, as
   illustrated in~\cref{fig:rk4_15}.  
  On the other hand, $\ell_3'$ is contained in $T_{p_4} M$ and
  $\ell_3$ is tangent to the sphere, from where we deduce that
  $\ell_4$ arrives tangentially to $\ell_3$ at $p_4$ and $T_{p_4}
  M=P_{O,p_3,p_4}$ (it contains $\ell_3$).
\end{remark}

\begin{figure}[htb!]
    \centering
    \includegraphics[width=0.6\linewidth]{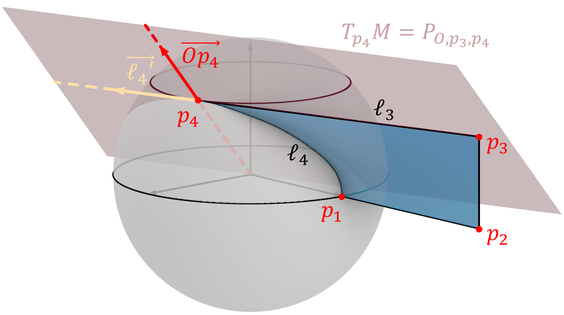}
    \caption{$\ell_4$ arrives tangentially to $\ell_3$ at $p_4$, because the tangent vector of $\ell_4$ at $p_4$ is tangent to $\mathbb{S}^2$ and must lie in the tangent plane $T_{p_4}M$ that contains $\ell_3$. }
    \label{fig:rk4_15}
\end{figure}

\begin{remark}\label{T}
  Since $M$ lies below $P_{O,p_3,p_4}$, the vertical projection
  $\pi(\ell_4)$ of $\ell_4$ over $\{z=0\}$ lies outside the
  simply-connected component bounded by the projection of the geodesic
  $\gamma$ (we are using the notation in the proof above).  On the
  other hand, the second item of Lemma~\ref{l:ConvexHull} says that
  $\pi(\ell_4)$ lies in the simply connected component bounded by the
  projection of the great circle passing through $p_1$ and $p_4$, so
  $\ell_4$ has little room to move (see~\cref{fig:l4}).
\end{remark}

\begin{figure}[htb!]
    \centering
    \includegraphics[width=\textwidth]{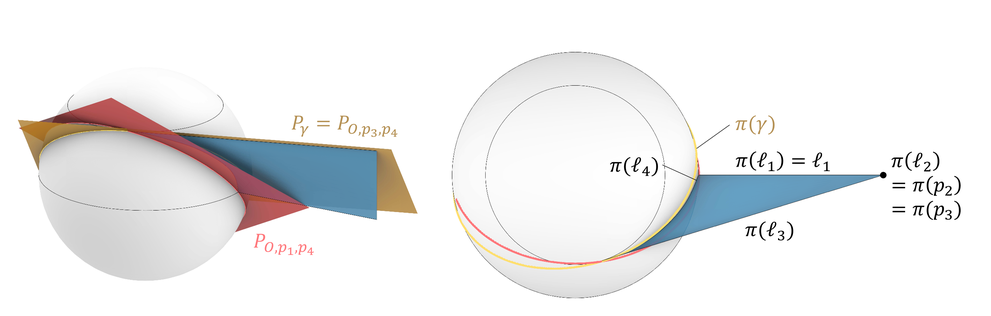}
    \caption{Region of $\{z=0\}$ where $\ell_4$ projects.}
    \label{fig:l4}
  \end{figure}

  \begin{lemma}\label{l:rotation}
  The minimal disk $M$ is a rotational graph, in the sense that it
  does not intersect any of its rotated copies about the $z$-axis, and the boundary curve $\ell_4$ has monotone third coordinate.
\end{lemma}

\begin{proof}
  Let us call $M_\theta$ the surface obtained by rotating $M$ by a
  (positive) angle $\theta$ about the $z$-axis. Since $M$ is contained
  in the wedge region $W_z$ with angle $\frac{m-2}{2m}\pi$ at the
  $z$-axis, if $\theta\geq \frac{m-2}{2m}\pi$ then $M_\theta$ does not intersect $M$. Decrease $\theta$ until $M_\theta$ intersects $M$ for the first time. Suppose $\theta>0$ and denote by $p$ a point where both surfaces meet.
By the maximum principle, $p$ cannot be interior to both $M$ and $M_\theta$.

  By the maximum principle at the boundary, $p$ cannot lie on both boundary curves in $\mathbb{S}^2$ (where the surfaces are free-boundary).  Moreover $M,M_\theta\subset\R^3-\B$, thus $p$
  cannot lie on $\ell_4$ and $M_\theta$ (nor on $M$ and the rotated copy of $\ell_4$).

  We now use Lemma~\ref{l:ConvexHull}: Since
  $M,M_\theta\subset\{0<z<h\}$, $p$ cannot lie on $\ell_1,\ell_3$ nor
  its rotated copies.  Since $M$ is contained on one side of the
  vertical plane $\{x=0\}$ and the rotated curve $\ell_2$ is contained
  on the other side, the only possibility left is $p$ being
  interior to $M_\theta$ and contained in $\ell_2$.  We now rule out
  this possibility: Since $M$ is contained on one
  side of the vertical plane $P_{\ell_3,z}$ containing $\ell_3$, its rotated copy cannot intersect $\ell_2$.

  We then reach a contradiction, so it must be $\theta=0$, and the
  lemma follows.
\end{proof}

  \begin{corollary}
    The solution $M$ to the partially-free boundary problem for the
    configuration $\langle\Gamma,\mathbb{S}^2\rangle$ (with $\Gamma$
    defined as above) is unique.
\end{corollary}

\begin{proof}
  Let us suppose there exist two such solutions $M$ and $\tilde M$. By
  Lemma~\ref{l:rotation}, they are both rotational graphs about the $z$-axis. 

  Let us rotate $M$ about the $z$-axis until it does not touch
  $\tilde M$ (it suffices to rotate by an angle $\frac{m-2}{2m}\pi$) and
  start rotating it back until it touches $\tilde M$ for the first
  time. By the maximum principle, either we arrive at its original
  position, or the rotated surface $M$ and $\tilde M$ intersect
  along their boundary curves on $\mathbb{S}^2$. In the latter case we
  reach a contradiction with the maximum principle at the
  boundary. Thus $M$ lies on one side of $\tilde M$.

  A symmetric argument shows that $M$ lies on the other side of $\tilde M$, so
  they must coincide.
\end{proof}

\begin{proposition}\label{p:z}
  The minimal disk $M$ is a vertical graph over a domain
  $\Omega\subset \{z=0\}$ and the vertical projection of $\ell_4$ is a
  radial graph from the origin.
\end{proposition}

\begin{proof}
  It suffices to check that the hypotheses of
  Proposition~\ref{p:graph} and Lemma~\ref{l:radial} are satisfied for $\Pi=\{z=0\}$, as illustrated in
 ~\cref{fig:graph_config2}. By construction, $\Gamma$ is the
  union of three straight segments with independent directions and the
  vertical plane passing through $p_1$ and $p_4$ separates $\Gamma$
  from the origin. Then condition $(c)$ is satisfied. By
  Remark~\ref{r:convex}, condition $(a)$ holds. We get $(b)$ from
 Lemma~\ref{l:ConvexHull} (we take $W=W_z$).

  We finish by checking that $(d)$ is also
  satisfied:  From Lemma~\ref{l:ConvexHull}, we get that the
  (embedded) boundary curve $\ell_4$ lies in the North hemisphere of
  $\s^2$, that projects graphically on $\{z=0\}$. Hence $\ell_4$
  projects one-to-one on the curve $\pi(\ell_4)\subset W_z\cap \{z=0\}$ with
  endpoints $\pi(p_1)=p_1$ and $\pi(p_4)$.  Since $\ell_2$ projects
  onto $\pi(p_2) = \pi(p_3)$ and $M$ lies on one side of $\{x=0\}$ and
  on one side of the vertical plane $P_{\ell_3,z}$, we deduce that the
  curve defined by $\pi(\ell_1) \cup \pi(\ell_3) \cup \pi(\ell_4)$ is
  a Jordan curve in $\{z=0\}$ which bounds a compact domain $\Omega$,
  see~\cref{fig:omega_proj}. This is the domain over which $M$
  projects graphically.
  \end{proof}

\begin{figure}[htb!]
    \centering
    \includegraphics[width=0.6\linewidth]{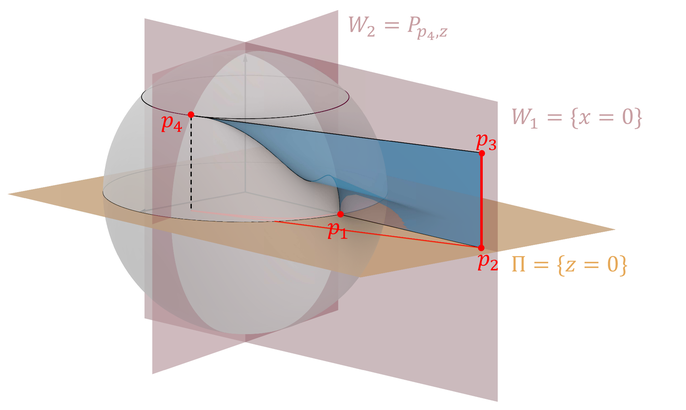}
    \caption{The minimal surface is contained in the wedge delimited by the two red planes $ \{x=0\}$ and $P_{p_4,z}$, and projects injectively on the yellow plane $\{z=0\}$. The segment $\ell_2$ from $p_2$ to $p_3$ (shown in red) projects on a point.}
    \label{fig:graph_config2}
\end{figure}

\begin{figure}[htb!]
    \centering
    \includegraphics[width = 0.4\textwidth]{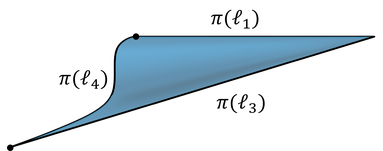}
    \caption{The curve $\pi(\ell_1) \cup \pi(\ell_3) \cup \pi(\ell_4)$ bounds a compact domain $\Omega\subset\{z=0\}$.}
    \label{fig:omega_proj}
\end{figure}

Since $M$ is a vertical graph, we can fix the Gauss map $N$ of $M$
pointing upwards. We observe that $N$ is horizontal along $\ell_2$ and
we can consider the angle it forms with a fixed horizontal
direction. We are going to prove that such an angle provides a
monotone function along $\ell_2$.

\begin{corollary}\label{l:N}
  There are no two points in $\ell_2$ with the same Gauss map.
\end{corollary}

\begin{figure}[htb!]
    \centering
    \includegraphics[width=0.5\linewidth]{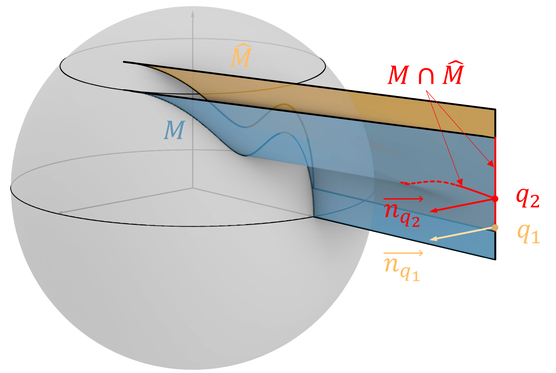}
    \caption{If there are two points $q_1, q_2$ with the same Gauss map, then the intersection of the minimal surface with its vertically translated copy (mapping $q_1$ to $q_2$) is a curve, meaning that interior points are projected onto the same point, which contradicts the fact that the minimal surface is a vertical graph.}
    \label{fig:lemma4_12}
\end{figure}

\begin{proof}
  Suppose there exist two different points $q_1,q_2\in\ell_2$ with the same Gauss map. Let us call $\hat M$ the translated copy of $M$ mapping $q_1$ to $q_2$, as illustrated in~\cref{fig:lemma4_12}. We observe that a neighborhood of $q_2$ in $\ell_2$ is contained in the boundary of $\hat M$. Since $\hat M$ and $M$ are tangent at $q_2$, we obtain using the maximum principle at the boundary that there must exist a curve contained in $\hat M\cap M$ emanating from $q_2$ (\textit{i.e.} $\hat M$ cannot be contained on one side of $M$). As $\hat M$ is a vertical translation of $M$, we obtain that there are different points in the interior of $M$ with the same vertical projection, contradicting Proposition~\ref{p:z}.
\end{proof}

The following proposition proves that $M$ is also a graph when
projecting orthogonally in the direction of $\ell_1$ and $\ell_3$. We
obtain, as a consequence arguing as in Corollary~\ref{l:N}, that
there are no two points in $\ell_1$ nor in $\ell_3$ with the same
Gauss map.  We observe that, since $\ell_3$ is tangent to
$\mathbb{S}^2$, then the plane perpendicular to $\ell_3$ passing
through the origin coincides with the vertical plane $P_{p_4,z}$
passing through $p_4$ and the origin.

\begin{proposition}\label{p:y}
  The minimal disk $M$ is a horizontal graph over $\{y=0\}$ and over $P_{p_4,z}$. As a consequence, there are no two points in $\ell_1$ nor in $\ell_3$ with the same Gauss map.
\end{proposition}

\begin{figure}[htb!]
    \centering
    \includegraphics[width=0.6\linewidth]{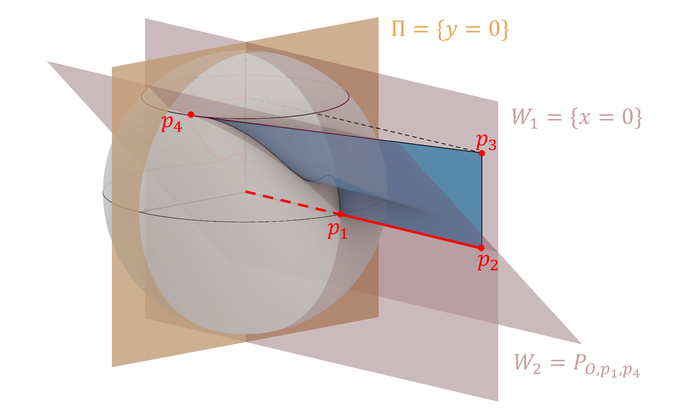}
    \caption{The minimal surface is contained in the wedge delimited by the two red planes $ \{x=0\}$ and $P_{O,p_1,p_4}$, and projects injectively on the yellow plane $\{y=0\}$. The segment $\ell_1$ from $p_1$ to $p_2$ (shown in red) projects onto a point.}
    \label{fig:graph_config1}
\end{figure}

\begin{figure}[htb!]
    \centering
    \includegraphics[width=0.6\linewidth]{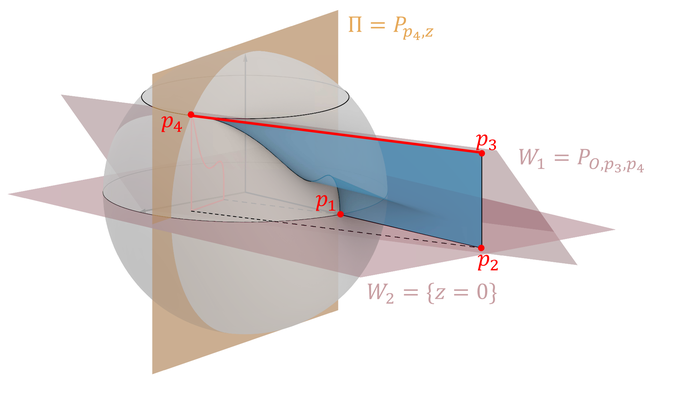}
    \caption{The minimal surface is contained in the wedge delimited by the two red planes $ \{z=0\}$ and $P_{O,p_3,p_4}$, and projects injectively on the yellow plane $P_{p_4,z}$. The segment $\ell_3$ from $p_3$ to $p_4$ (shown in red) projects onto a point.}
    \label{fig:graph_config3}
\end{figure}

\begin{proof}
  We are going to use Proposition~\ref{p:graph} taking $\Pi$ as
  $\{y=0\}$ (see~\cref{fig:graph_config1}) or $P_{p_4,z}$ (see
 ~\cref{fig:graph_config3}).  By Remark~\ref{r:convex},
  hypothesis $(a)$ holds in both cases.
  
  From Lemma~\ref{l:ConvexHull} we know that $M$ is contained in the
  wedge region $W_z$ and hence it is located on one side of the planes
  $\{x=0\}$ and $P_{p_4,z}$. Moreover
  $\ell_4 \subset (\mathbb{S}^2\cap W_z)$ is embedded, so $\ell_4$
  projects one to one over $\{y=0\}$ (resp. $P_{p_4,z}$). Then, up to
  $\ell_1$ (resp. $\ell_3$), $\Gamma\cup\ell_4$ projects injectively
  on $\{y=0\}$ (resp. $P_{p_4,z}$), and its projection bounds a
  domain. Therefore, hypothesis $(d)$ is satisfied.

  Again by Lemma~\ref{l:ConvexHull}, $M$ lies above $P_{O,p_1,p_4}$
  and on one side of $\{x=0\}$. These two planes intersect along the
  $y$-axis and they are both orthogonal to $\{y=0\}$ and $M$ is contained
  in one of the convex wedges determined by these planes. This is the
  wedge region we consider satisfying $(b)$ for the case
  $\Pi=\{y=0\}$. The plane $P_{O,p_1,p_4}$ is the one appearing in
  $(c)$; it contains the origin and leaves $\Gamma$ on one
  side in a non-strictly way, and we are done in this case.

  In the case of $\Pi=P_{p_4,z}$, we consider the convex wedge region
  between $\{z=0\}$ and $P_{0,p_3,p_4}$ that contains $M$ (see
  Lemma~\ref{l:ConvexHull}). This wedge satisfies condition $(b)$ of
  Proposition~\ref{p:graph}.  Finally, we observe that the plane
  orthogonal to $P_{p_4,z}$ passing through $p_1,p_4$ must contain
  $\ell_3$ and its parallel line passing through $p_1$ (at height
  zero). In particular, this plane separates $\Gamma$ from the origin,
  and condition $(c)$ is also satisfied.
\end{proof}

Finally, we are going to prove that the vertical coordinate of the normal vector along $\ell_4$ never vanishes.

\begin{lemma}\label{l:Nhoriz}
  The Gauss map $N$ is never horizontal at points of $\ell_4$.
\end{lemma}

\begin{proof}
  Suppose there is a point $q\in \ell_4$ whose tangent plane $P$ is
  vertical, as shown in~\cref{fig:lemme4.14}. We recall that the
  tangent plane of $M$ at any point of $\ell_4$ is the plane defined
  by the position vector and the tangent vector $\ell_4'$ at the
  point. Thus $P$ passes through the origin.

  Since $P$ and $M$ have contact order at least one, $P$ intersects
  $\overline M$ in two or more curves emanating from $q$. By
  Lemma~\ref{l:ConvexHull} we know that $P$ must intersect $W_z$, and
  it intersects $\Gamma$ at a single point.  By the maximum principle,
  $\overline M\cap P$ cannot contain any closed loop. In particular,
  one of the curves in $\overline M\cap P$ arrives at a point $q'$ in
  $\ell_4$, $q'\neq q$.  This contradicts Proposition~\ref{p:z}
  because the vertical projection of $\ell_4$ is a radial graph from
  the origin.
\end{proof}

\begin{figure}[htb!]
    \centering
    \includegraphics[width=0.5\linewidth]{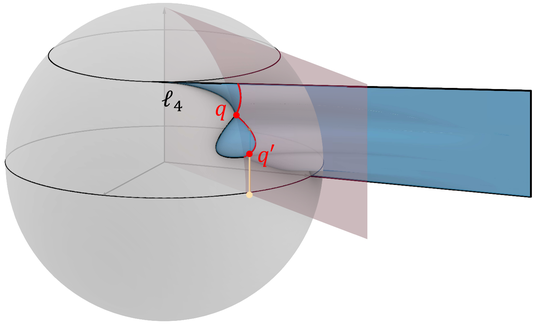}
    \caption{If the Gauss map is horizontal at some point of $\ell_4$,
      then one of the intersection curves must end at another point
      $q'\in\ell_4$, which contradicts that the vertical projection of
      $\ell_4$ is a radial graph from the origin.}
    \label{fig:lemme4.14}
\end{figure}

We now have all the ingredients to prove Theorem \ref{th:cat}.

\begin{th*}
  For any $m\geq 3$, there exists a one-parameter family (after identifying by isometries of $\mathbb{R}^3$) of embedded minimal annuli, called {\rm tensile catenoids}. Any boundary curve of one such surface is the union of $m$ arcs, any one of them an asymptotic line of the surface with constant curvature. Two consecutive boundary arcs meet, forming a cusp (\textit{i.e.} they meet asymptotically, forming a corner on the surface of intrinsic angle zero). Moreover, any annulus in this family is symmetric with respect to a horizontal plane dividing the surface into two vertical graphs; and it also has $m$ vertical planes of symmetry meeting at an angle $\frac{\pi}{m}$.

  The parameter of the family is the length $\rho$ of the neck of the tensile catenoids (\textit{i.e.} the length of the intersection of each annulus with its horizontal plane of symmetry), going from zero (the limit when $\rho$ goes to zero is a piece of the horizontal plane of symmetry) to $2m\cos\frac{\pi}{m}$ (in this case, the annulus splits into $m$ disks, two consecutive ones joined by a point in their boundary).
\end{th*}

\begin{proof}
  Let us consider the surface $M$ obtained above for some fixed $m\geq
  3$ and $h\in(0,\cos\frac\pi m)$.  We call $M^*$ the conjugate
  minimal surface of $M$. Since the Gauss map $N^*$ of $M^*$ coincides
  with $N$ (the Gauss map of $M$ considered above) at conjugate
  points, it always points upwards, and $M^*$ is a vertical
  multigraph. We are going to prove that $M^*$ is in fact a graph (in
  particular, embedded). We observe that the projection of $M$ over
  the horizontal plane $\{z=0\}$ 
   is not convex, so Krust's Theorem cannot be applied. Instead, we are going to
  study the behaviour of the boundary curves of $M^*$, which will
  project one-to-one on the horizontal plane, and we will conclude the
  graphical property of $M^*$.

  Let us call $\ell_i^*\subset\partial M^*$ the curves obtained by
  conjugation from $\ell_i\subset\partial M$, for any $i$. The
  correspondence between conjugate curves $\ell_i$ and $\ell_i^*$ is
  illustrated in~\cref{fig:thm4_7_proof}.

\begin{figure}[htb!]
    \centering
    \includegraphics[width=\linewidth]{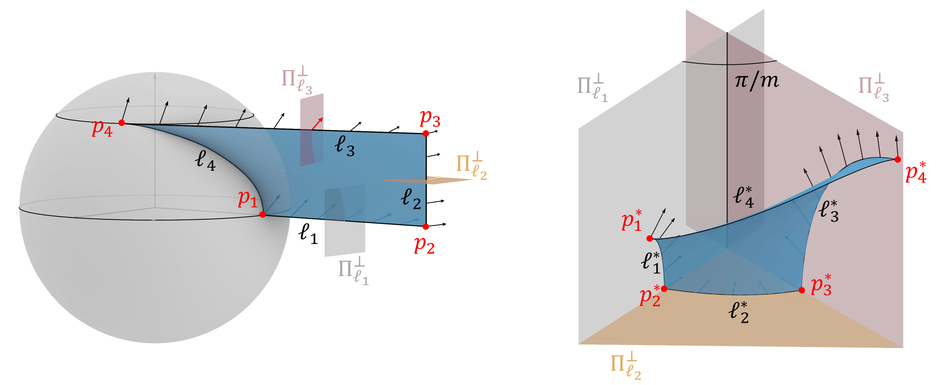}
    \caption{Conjugate boundaries: straight lines $\ell_i$ have their
      Gauss map in a plane $\Pi_{\ell_i}^{\perp}$ perpendicular to $\ell_i$.}
    \label{fig:thm4_7_proof}
\end{figure}

  Since $\ell_2$ is a vertical straight line (in particular, an
  asymptotic curve), $\ell_2 ^*$ is a curvature line of symmetry
  contained in a horizontal plane ($N^*$ is horizontal along~$\ell_2$)
  that can be assumed to be $\{z=0\}$ up to a vertical
  translation. Then $M^*$ can be extended by mirror symmetry with
  respect to that plane.  By Corollary~\ref{l:N} we know that the angle of
  $N^*$ along $\ell_2^*$ is a strictly monotone function, giving local
  convexity for $\ell_2 ^*$. Moreover, the total variation of the
  angle of $N$ along $\ell_2$ (and then of $N^*$ along $\ell_2^*$)
  coincides with $\frac\pi m$, the angle between the vertical planes
  $\{x=0\}$ and~$P_{\ell_3,z}$. In particular, we obtain that
  $\ell_ 2^*$ is embedded.

  We call $\Omega^*\subset\{z=0\}$ the Alexandrov-embedded domain over
  which $M^*$ is a multigraph. We know that $M^*$ arrives orthogonally
  to $\{z=0\}$ along $\ell_2^*$. By the maximum principle with
  vertical planes we deduce that $\ell_ 2^*$ is concave with respect
  to $\Omega^*$. (Since $N^*$ points upwards on $M$, it coincides with
  the outer conormal to $\Omega^*$ along $\ell_ 2^*$.)

  Arguing similarly using Proposition~\ref{p:y} we also obtain that
  $M^*$ is a multigraph over a domain   $\Omega_1^*\subset \Pi_{\ell_1}^{\perp}$
  (resp.   $\Omega_3^*\subset \Pi_{\ell_3}^{\perp}$) and that up to an isometry,
 $\ell_1^*\subset \Pi_{\ell_1}^{\perp}$(resp.     $\ell_3^*\subset \Pi_{\ell_3}^{\perp}$) is an
  embedded line of curvature concave with respect to the projection
  $\Omega_ 1^*$ (resp. $\Omega_3^*$) of $M$ over 
 $\Pi_{\ell_1}^{\perp}$
  (resp.~$\Pi_{\ell_3}^{\perp}$). Moreover, $M^*$ can be extended by mirror
  symmetry about the vertical plane $\Pi_{\ell_1}^{\perp}$ (resp.~$\Pi_{\ell_3}^{\perp}$). 
  (The fact that the total variation of the angle of $N^*$ along $\ell_2^*$
  coincides with the angle $\frac \pi m$ between the vertical planes $\{x=0\}$
  and~$P_{\ell_3,z}$ justifies that we may take
  $\Pi_{\ell_1}^{\perp}=\{y=0\}$ and $\Pi_{\ell_3}^{\perp}=P_{p_4,z}$.)
  
  Finally, since $M$ and $M^*$ are isometric, we know that both
  $\ell_1^*$ and $\ell_3^*$ intersect $\ell_2^*$ orthogonally, and we
  can assume they both lie in $\{z>0\}$ near $\ell_2^*$. Moreover, $\ell_1^*$ and
  $\ell_3^*$ start projecting to the concave side of $\ell_2^*$. Since
  along $\ell_1^*$ (resp. $\ell_3^*$) the Gauss map $N^*$ is contained
  in the vertical plane containing the curve, it is monotone, its
  total variation is less than $\frac \pi 2$, and it becomes
  horizontal only when intersecting $\ell_2^*$, we deduce that $\ell_1^*,\ell_3^*\subset \{z>0\}$ and that the
  vertical projection of $\ell_1^*$ (resp. $\ell_3^*$) is one-to-one.

It just remains to study the behaviour of the curve $\ell_4^*$.  We recall that $\ell_4$ is a line of curvature of $M$ located on the sphere, with $\tau_n=0$ and $\kappa_g=1$. Hence $\ell_4^*$ is an asymptotic curve of $M$ ($\kappa_n^*=0$) and has constant geodesic curvature $\kappa_g^*=1$. Therefore, $\ell_4^*$ has constant curvature and the curvature vector of $\ell_4^*$ coincides with the outer conormal of $M$ along $\ell_4^*$. Moreover, the vertical coordinate of the normal $N$ along $\ell_4$ (and so of $N^*$ along $\ell_4^*$) never vanishes by Lemma~\ref{l:Nhoriz}. Thus the projection of $\ell_4^*$ is concave with respect to the Alexandrov-embedded domain $\Omega^*$.

Let $W^*$ be the open convex wedge region bounded by $\Pi_{\ell_1}^{\perp}$ and $ \Pi_{\ell_3}^{\perp}$. There are neighborhoods of $\ell_1^*$ and $\ell_3^*$ where the surface $M^*$ is contained in $W^*$ (as these lines are planar lines of curvature and the surface is contained in $W^*$ close to $p_2^*$ and $p_3^*$). In particular, there exist neighborhoods of $p_1^*$ and $p_4^*$ where the surface $M^*$ is contained in $W^*$. Then $\ell_4^*$ is contained in $W^*$ near its endpoints.

\begin{figure}[htb!]
    \centering
    \includegraphics[width = \textwidth]{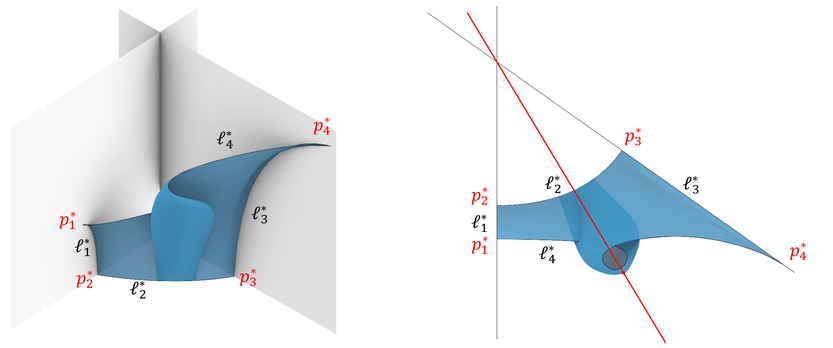}
    \caption{This picture shows the convex wedge region  $W^*$ that contains $M^*$ in
      neighborhoods of $\ell_1^*$ and $\ell_3^*$. The red line
      corresponds to the vertical projection of the vertical plane
      $P$. In this picture the farthest point of $M^*$ from the
      $z$-axis would be interior. }
    \label{fig:concave}
  \end{figure}
  
Now take a vertical plane $P$ passing through the origin and intersecting $W^*$, as illustrated in~\cref{fig:concave}. We consider the farthest point $q$ of $M^*\cap P$ from the $z$-axis. We observe that this point cannot be interior to $M^*$, as the Gauss map at that point would be horizontal, in contradiction with the fact that $M^*$ is a vertical multigraph. This point $q$ cannot lie in $\ell_2^*$ either, as $\ell_2^*$ is concave with respect to $\Omega^*$ and it arrives orthogonally to $ \Pi_{\ell_1}^{\perp}$ and $ \Pi_{\ell_3}^{\perp}$.  Hence $q$ must be a point in~$\ell_4^*$. 

Let us prove there cannot be more points in $P\cap\ell_4^*$: suppose there exists a different point $q'$ in $P\cap\ell_4^*$. Since the vertical projection $\pi(\ell_4^*)$ of $\ell_4^*$ is concave with respect to $\Omega^*$, in order to go from $\pi(q)$ to $\pi(q')$ along $\pi(\ell_4^*)$ we should intersect $P$ at a point farther away from the origin than $\pi(q)$, a contradiction.  Therefore, we obtain that any vertical plane in $W^*$ containing the $z$-axis necessarily intersects  $\ell_4^*$ in exactly one point. We conclude that $\ell_4^*\subset W^*$ must be embedded with injective vertical projection, and $M^*\subset W^*$ is a vertical graph.

To finish the proof of the existence of the corresponding tensile
catenoid, we only need to extend $M^*$ by reflection symmetry on the
planes containing the boundary curves $\ell_1^*,\ell_2^*$ and
$\ell_3^*$ successively. To obtain the embeddedness of the tensile catenoid it remains to prove $M^*\subset\{z>0\}$. Since $\Gamma^*\subset\{z\geq 0\}$, it suffices to prove $\ell_4^*\subset\{z>0\}$. Assume that $M^*\subset\{z<0\}\neq \emptyset$. By the maximum principle, the function $p \to z(p)$ takes its minimum value at a boundary point $p^*\in\ell_4^*$ and at this point the third coordinate of the curvature vector of $\ell_4^*$ is negative (i.e. $h(p^*):=\langle \vec k(p^*),e_3\rangle<0$). Since $h(p_1^*)>0$ and the function $h$ changes sign along $\ell_4^*$, there is a point $q^*_0\in\ell:4^*$ where $h(q^*_0)=0$ and $h$ changes sign around $q^*_0$. By conjugacy properties, at the point $q_0\in\ell_4$ corresponding to $q^*_0$, the tangent vector of $\ell_4\subset M$ is horizontal (rotating by $\pi$ the curvature vector into the tangent plane gives a tangent vector of $\ell_4$). Hence the horizontal plane $P(q_0)$ passing through $q_0$ is tangent to $\ell_4 \subset M$. Since the third coordinate of the curvature changes sign, there is a subarc of $\ell_4$ where the third coordinate of the tangent vector changes sign. Hence the height of the curve $\ell_4$ has a local maximum at $q_0$. In particular this implies that $\ell_4$ is not a rotational graph, contradicting the Lemma \ref{l:rotation}.

Finally, we observe that by Remark~\ref{r:cusp}, $\ell_4^*$ and $\ell_3^*$ meet tangentially, which creates cusps in the final geometry of the boundary of the corresponding tensile catenoid.

We also know that the length of $\ell_2^*$ is $h\in(0,\cos\frac{\pi}{m})$, the length of $\ell_1^*$ is $a=\sqrt{1-h^2}\csc\frac{\pi}{m}-1$ and the length of $\ell_3^*$ is $\sqrt{(1+a)^2-(1-h^2)}=\sqrt{1-h^2}\,\cot\frac{\pi}{m}$. We observe that these lengths are not invariant under homotheties of the space, so an annulus in the family cannot be obtained from another by applying a homothety. Thus $h$ is a parameter in the family. The neck length of the corresponding tensile catenoid is equal to $\rho=2mh$, so $\rho$ can also be considered as the family parameter.

When $h$ goes to zero, $M$ is contained in the horizontal plane $\{z=0\}$, and the tensile catenoids converge to the double covering of a piece of the plane bounded by $m$ circular arcs that meet tangentially. As $h$ goes to $\cos\frac{\pi}{m}$, the length of $\ell_1^*$ goes to zero, and the tensile catenoids converge to the union of $m$ minimal disks bounded by four arcs of constant curvature, two consecutive ones joined by a common vertex.
\end{proof}

\subsection{Tensile $k$-noid}
\label{sub:k}

In this subsection we will prove, for any fixed $k\geq 3$, the
existence of the so-called tensile $k$-noids appearing in
Theorem~\ref{th:k}, that is, minimal surfaces with the topology of the
$k$-noids, each boundary component of which  consists of
four asymptotic lines of constant curvature joined by cusps.
This surface is illustrated in~\cref{fig:knoid_symetries}.

\begin{th**}
  For any $k\geq 3$, there exists an embedded minimal surface with the topology of a sphere minus $k$ disks, called {\rm tensile $k$-noid}. Any boundary curve of this surface is the union of four arcs, any one of them an asymptotic line of the surface with constant curvature, two consecutive ones meeting forming a cusp (\textit{i.e.} they form a corner on the surface of angle zero). Moreover, the tensile $k$-noid is symmetric with respect to a horizontal plane dividing the surface into two vertical graphs; and it also has $k$ vertical planes of symmetry meeting at an angle $\frac{\pi}{k}$.
\end{th**}

\begin{figure}
    \centering
    \includegraphics[width=0.8\linewidth]{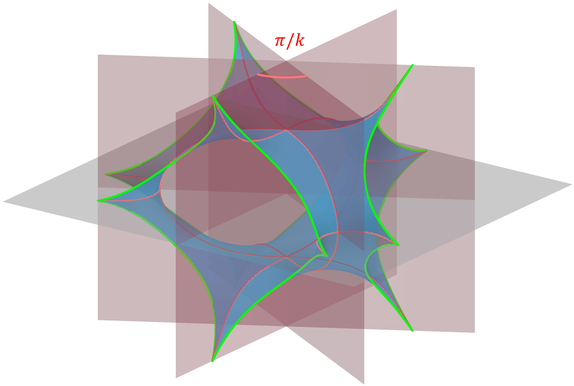}
    \caption{$k$-noid with $k=3$ boundary components.}
    \label{fig:knoid_symetries}
\end{figure}

As in the case of the tensile catenoids, these examples will be constructed by symmetry from a fundamental piece obtained as the conjugate surface of the solution to a partially-free boundary problem. Let us start by describing that problem.

Fix $k\geq 3$. For any $a>1$, let us now consider the points $p_1=(0,0,-1)$,
$p_2=(0,a,-1)$, $p_3=(\cos(\frac \pi k), \sin(\frac \pi k),-1)$ and
$p_4=(\cos(\frac \pi k), \sin(\frac \pi k),0)$. We call $\ell_i$, for
$i\in\{1,2,3\}$, the straight segment joining $p_i$ to
$p_{i+1}$. Taking $a=\csc\frac\pi k$, the segment $\ell_2+(0,0,1)$
arrives tangentially to the equator of $\s^2$ and $\ell_3$ arrives
tangentially to the unit sphere, see~\cref{fig:configuration_k}. We call $\Gamma=\ell_1\cup
\ell_2\cup \ell_3$.

We observe that this construction does not work for $k=2$, so no tensile catenoid can be provided with this construction.

\begin{figure}[htb!]
  \centering \includegraphics[width=\textwidth]{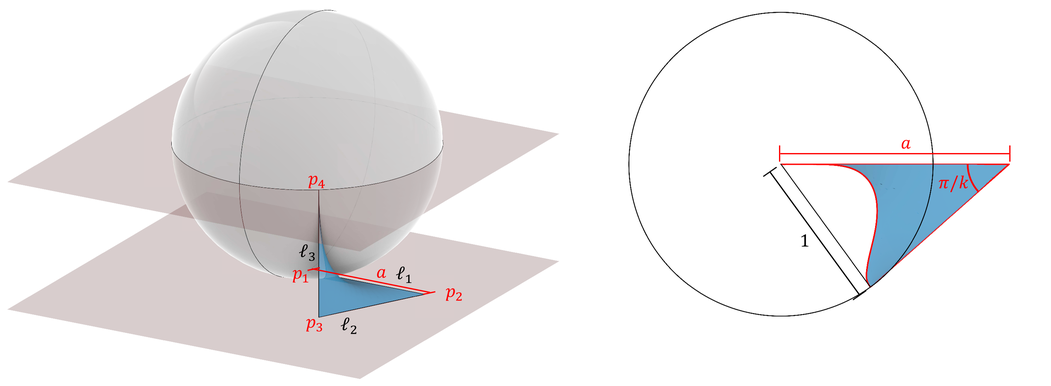}
  \caption{The piecewise smooth curve $\Gamma=\ell_1\cup \ell_2\cup \ell_3$ and its vertical projection in $\{z=0\}$.}
  \label{fig:configuration_k}
\end{figure}

By Theorem~\ref{th:semifree}, there exists an embedded minimal surface $M$ outside $\B$ with $\Gamma\subset\partial M$ and so that $M$ meets~$\s^2$ orthogonally along the smooth embedded curve $\ell_4 = \partial M-\Gamma\subset\s^2$. Moreover, $\ell_4$ is a line of curvature of $M$ with constant curvature (see Claim~\ref{cl:curvatureline}). We follow the steps in subsection~\ref{sub:cat}, so we start by describing the region containing the surface~$M$.

\begin{lemma}\label{l:ConvexHull_k}
For any $i\in\{1,2\}$ we denote by $P_{\ell_i,z}$ the vertical plane containing $\ell_i$, by $O$ the origin and we call $P_{a,c,b}$ the plane passing through the points $a,b$ and $c$.
\begin{enumerate}
    \item The interior of both $M$ and $\ell_4$ are contained in
      $W_z\cap\{-1<z<0\}$, with $W_z$ the (open convex) wedge region from
      $P_{O,p_3,p_4}$  to
      $P_{\ell_1,z}=\{x=0\}$ containing only points with positive $x$ and $y$ coordinates, as illustrated in~\cref{fig:knoid_wedge}.
    \item $M$ lies on one side of the vertical plane $P_{\ell_2,z}$.
    \item $M$ lies on one side of the vertical plane $\{x=\cos\frac\pi k\}$ containing $\ell_3$.
\end{enumerate}
\end{lemma}

\begin{figure}[htb!]
    \centering
    \includegraphics[width=0.7\linewidth]{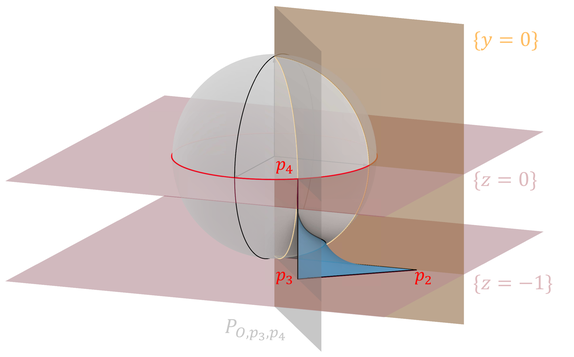}
    \caption{The piecewise smooth curve $\Gamma=\ell_1\cup\ell_2\cup\ell_3$ together with a curve $\ell_4$ in $\s^2$ joining the
endpoints of $\Gamma$ and a wedge containing $M$ appearing in Lemma~\ref{l:ConvexHull_k}.
}
    \label{fig:knoid_wedge}
\end{figure}

\begin{proof}
  First we observe that $\Gamma$ is contained in the closure of $W_z\cap\{z<0\}$. Then Lemma~\ref{l:CHull} applied to $\{z=0\}$, $P_{O,p_3,p_4}$ and $\{x=0\}$, ensures that the interior of both $M$ and $\ell_4$ are contained in $W_z\cap\{z<0\}$. We also observe that $\partial M\subset \Gamma\cup\s^2 \subset\{z\geq -1\}$ and that $\partial M$ lies on one side of~$P_{\ell_2,z}$, so we conclude items {\it 1} and {\it 2} by using the maximum principle.
Since $\ell_4\subset\s^2\cap W_z$ is contained in the same side of the vertical plane $\{x=\cos\frac\pi k\}$ as $\Gamma$, we deduce item {\it 3} by the maximum principle.         
\end{proof}

\begin{remark}\label{rem:T}
  The convex hull of the vertical projection of $p_1,p_2,p_3,p_4$ over $\{z=0\}$ is a triangle that contains the vertical projection of $M$ by Lemma~\ref{l:ConvexHull_k}. In particular, the vertical projection of $\ell_4$ lies in the interior of $T$.
\end{remark}

Using the previous lemma, we now prove as in Subsection~\ref{sub:cat} that $M$ is a graph in the directions orthogonal to the segments in $\Gamma$. We observe that, since $\ell_2+(0,0,1)$ arrives tangentially to the equator of $\s^2$, we get that $P_{O,p_3,p_4}$ coincides with the vertical plane orthogonal to $\ell_2$ and passing through the origin.

\begin{proposition}\label{p:z_k}\ 
  The minimal disk $M$ is an orthogonal graph over $\{y=0\}$ (resp. $P_{O,p_3,p_4}$ and $\{z=0\}$), and there are no two points in $\ell_1$ (resp. $\ell_2$ and $\ell_3$) with the same Gauss map.
\end{proposition}

\begin{figure}[htb!]
    \centering
    \includegraphics[width=0.32\linewidth]{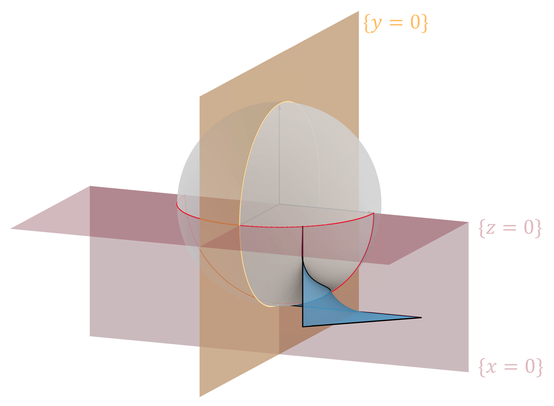}
    \hfill
    \includegraphics[width=0.32\linewidth]{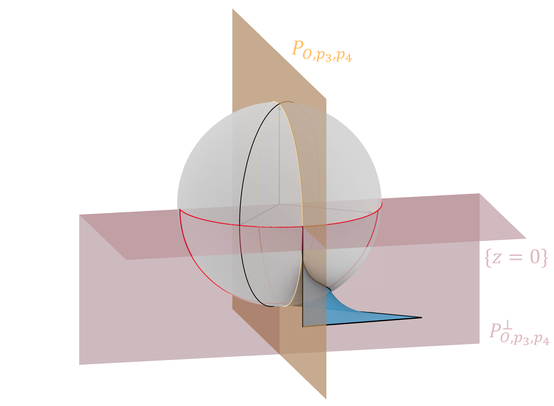}
    \hfill
    \includegraphics[width=0.32\linewidth]{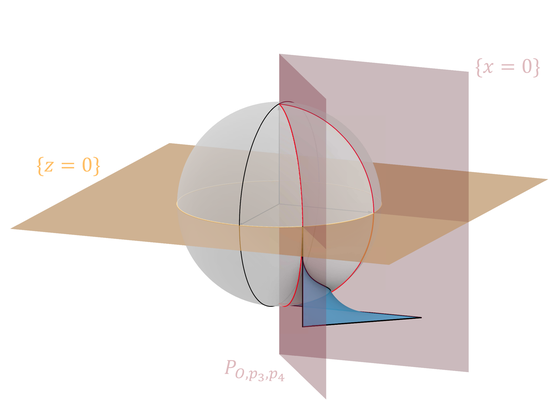}
    \caption{Perspective view of the planes delimiting the surface.}
    \label{fig:knoid_planes}
\end{figure}

\begin{proof}
  We use Proposition~\ref{p:graph} to prove this result, as the three
  planes pass through the origin. We have constructed $\Gamma$ as the
  union of three independent straight segments so that $(a)$ holds by Remark~\ref{r:convex}.

  Let us check that $(b)$ is satisfied in the three cases, illustrated in~\cref{fig:knoid_planes}:
  \begin{itemize}
  \item When $\Pi=\{z=0\}$, the wedge is $W_z$ defined in Lemma~\ref{l:ConvexHull_k}.
   \item By Lemma~\ref{l:ConvexHull_k}, $M$ lies below $\{z=0\}$ and on one side of $\{x=0\}$. Hence, for $\Pi=\{y=0\}$, consider the wedge containing $M$ defined by the other two coordinate planes.
   \item Let us call $P_{O,p_3,p_4}^{\perp}$ the vertical plane
     orthogonal to $P_{O,p_3,p_4}$ and passing through the origin. The
     intersection of $P_{O,p_3,p_4}^{\perp}$ with $\{z=0\}$ is the
     translated copy by vector $-p_3$ of the straight line containing
     $\ell_2$.  Since $W_z$ (defined in Lemma~\ref{l:ConvexHull_k}) lies
     on one side of $P_{O,p_3,p_4}^{\perp}$, one of the convex wedges
     determined by$P_{O,p_3,p_4}^{\perp}$ and $\{z=0\}$ contains
     $M$. This is a wedge satisfying hypothesis~$(b)$ of
     Proposition~\ref{p:graph} for $\Pi=P_{O,p_3,p_4}$.
 \end{itemize}

 The plane appearing in $(c)$ of Proposition~\ref{p:graph} for
 $\{z=0\}$ is $P_{O,p_3,p_4}$, that contains the
 origin.  For the plane $\{y=0\}$ (resp. $P_{O,p_3,p_4}$), the
 corresponding plane is $P_{p_1,p_2,p_4}$ (resp.
 $P_{p_1,\tilde p_2,p_4}$, with $\tilde p_2=p_2+(0,0,1)$); in both
 cases the plane leaves $\Gamma$ below and the origin above. Hence
 hypothesis $(c)$ is satisfied.

All that remains is to check that the hypothesis $(d)$ is fulfilled in
the three cases.  
By Lemma~\ref{l:ConvexHull_k}, $\ell_4$ is contained in a hemisphere of $\s^2$ on one side of $\{y=0\}$ (resp. $P_{O,p_3,p_4}$ and $\{z=0\}$), where the corresponding projection is injective. Hence, the orthogonal projection of $\ell_4$ over $\{y=0\}$ (resp. $P_{O,p_3,p_4}$ and $\{z=0\}$) is injective.

Since we know from Lemma~\ref{l:ConvexHull_k} that $M$ lies on one
side of the planes perpendicular to $\{y=0\}$ (resp. $P_{O,p_3,p_4}$; $\{z=0\}$) that contain $\ell_2$ and $\ell_3$ (resp.  $\ell_1$ and
$\ell_3$; $\ell_1$ and $\ell_2$), we get that up to $\ell_1$
(resp. $\ell_2$; $\ell_3$) the boundary of $M$ projects injectively
on $\{y=0\}$ (resp. $P_{O,p_3,p_4}$; $\{z=0\}$), and its projection
bounds a domain. Therefore, hypothesis $(d)$ is satisfied, and we
conclude that $M$ is an orthogonal graph over $\{y=0\}$
(resp. $P_{O,p_3,p_4}$; $\{z=0\}$).

We finish the proof of the proposition arguing as in Corollary~\ref{l:N}.
\end{proof}

\begin{remark}\label{r:cusp_k}
  Since the position vector at $p_4$ is contained in $T_{p_4} M$ (as $M$ is orthogonal to $\s^2$ along $\ell_4$) and $\ell_3$ arrives tangentially to $\s^2$ at $p_4$, we get that $T_{p_4} M=P_{O,p_3,p_4}$, the vertical plane containing $\ell_3$ and passing through the origin. Since $\ell_4'$ is tangent to $M$ at $p_4$ and orthogonal to the conormal of $M$ at $p_4$ (that is horizontal, by continuity along $\ell_4$), we deduce that $\ell_4$ arrives tangentially to $\ell_3$ at $p_4$.

  The same happens at $p_1$: Since the position vector at $p_1$ is contained in $T_{p_1} M$ (as $p_1$ is an endpoint of $\ell_4$) and $\ell_1$ is tangent to $\s^2$ at $p_1$, $T_{p_1} M= \{x=0\}$, and $\ell_4$ arrives tangentially to $\ell_1$ at $p_1$.
\end{remark}

Since $\Gamma$ can be separated by any plane in at most three
components and $\Gamma$ is contained on one side of
  $P_{O,p_1,p_4} = P_{O,p_3,p_4}$ by Lemma~\ref{l:ConvexHull_k}, Proposition~\ref{p:kg} says that the geodesic curvature of $\ell_4$ with respect to $\s^2$ never vanishes. Hence the geodesic curvature vector of $\ell_4$ in $\s^2$ points to the geodesic arc of $\s^2$ with endpoints $p_1,p_4$. Let us finally prove that the vertical coordinate of the Gauss map never vanishes along $\ell_4$.

\begin{lemma}\label{l:Nhoriz_k}
The Gauss map $N$ is never horizontal at points of $\ell_4$. 
\end{lemma}

\begin{proof}
Suppose there is a point $q\in \ell_4$ whose tangent plane $P$ (which we know passes through the origin) is vertical. In particular, the tangent vector of $\ell_4$ at $q$ is tangent to the great circle $\s^2\cap P$. By Remark~\ref{r:transverse}, $\ell_4$ has points on both sides of~$P$.

  Since $P$ and $M$ have contact order at least one, 
  $M\cap P$ contains at least two curves with $q$ as an endpoint. 
By
  Lemma~\ref{l:ConvexHull_k} we know that $P$ must intersect $W_z$, and
  it intersects $\Gamma$ at a single point.  By the maximum principle,
  $\overline M\cap P$ cannot contain any closed loop. In particular,
  one of the curves in $\overline M\cap P$ arrives at a point $q'$ in
  $\ell_4$, $q'\neq q$.

We observe that we cannot apply Lemma~\ref{l:radial} directly as $\partial M$ intersects the $z$-axis ($L_{W_z}$, with the notation in the lemma) at $p_1$. But since $\ell_1,\ell_4$ are asymptotic at $p_1$, $\ell_4$ intersects any half-plane in $W_z$ containing the $z$-axis, and the proof of Lemma~\ref{l:radial} also works
in this case. We conclude that the vertical projection of $\ell_4$ is a radial graph from the origin. Then $P$ can intersect $\ell_4$ at just one point, a contradiction.
\end{proof}

\medskip

\begin{figure}[htb!]
    \centering
    \includegraphics[width=\linewidth]{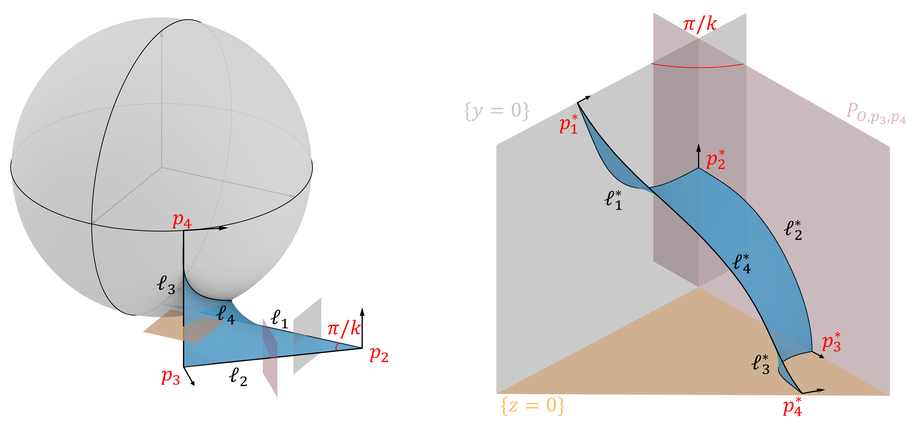}
    \caption{Conjugate surfaces and boundaries: straight lines $\ell_i$ have their Gauss map in a plane containing the line of curvature $\ell_i^*$.}
    \label{fig:knoid_proof_thm}
\end{figure}

Let us now prove Theorem~\ref{th:k}.~\cref{fig:knoid_proof_thm} illustrates the various steps of the proof.

\begin{proof}
  Let $M$ be the minimal surface constructed above for a fixed $k\geq 3$, and $M^*$ its conjugate surface. Since the Gauss map $N^*$ of $M^*$ coincides with the one of $M$ at conjugate points, we assume that it always points upwards by
  Proposition~\ref{p:z_k}. Hence $M^*$ is a multigraph over $\{z=0\}$ (also over $\{y=0\}$ and over $P_{O,p_3,p_4}$). We are going to prove that $M^*$ is in fact a vertical graph (in particular, embedded).

  Let us call $\ell_i^*\subset\partial M^*$ the curves obtained by conjugation from $\ell_i\subset\partial M$ for any $i$. As $\ell_1$ and $\ell_2$ are horizontal straight segments forming an angle $\frac\pi k$, then $\ell_1^*$ and $\ell_2^*$ are lines of symmetry contained in vertical planes forming an angle $\frac\pi k$, that can be assumed to be $\{y=0\}$ and $P_{O,p_3,p_4}$. 
  We also know that $\ell_3^*$ is a line of symmetry contained in a horizontal plane, which can be assumed to be $\{z=0\}$.

  Along $\ell_1^*$, the Gauss map is contained in $\{y=0\}$ and its third coordinate is strictly increasing (by Proposition~\ref{p:z_k}) coming from $0$ at $p_1^*$ to $1$ at $p_2^*$. Up to a possible symmetry on a horizontal plane, we can assume that the height function (the $z$ coordinate) is strictly decreasing along $\ell_1^*$ with this orientation. 

  When restricted to $\ell_2^*$, the Gauss map $N^*$ is contained in $P_{O,p_3,p_4}$ and its third component is strictly decreasing (Proposition~\ref{p:z_k}) coming from $1$ at $p_2^*$ ($N^*$ is vertical at $p_2^*$) to $0$ at $p_3^*$ ($\ell_2^*$ arrives orthogonally to $\{z=0\}$ at $p_3^*$). Thus, the height function is strictly monotone on $\ell_2^*$ with this orientation. If we extend $M^*$ by symmetry on the planes $\{y=0\}$ and $P_{O,p_3,p_4}$, we get that the height function cannot have a minimum on $p_2^*$ by the maximum principle, so the height function is also decreasing on $\ell_2^*$.

  The horizontal curve $\ell_3^*$ comes orthogonally from $P_{O,p_3,p_4}$ at $p_3^*$ and the angle that $N^*$ makes with some fixed horizontal direction is strictly monotone by Proposition~\ref{p:z_k}. The Gauss map $N^*$ and the direction of the straight segment $\ell_2$ are orthogonal at $p_3^*$ and tangent at $p_4^*$.
  
  We call $\Omega^*\subset\{z=0\}$ the Alexandrov-embedded domain over which $M^*$ is a multigraph. We know that $M^*$ arrives orthogonally to $\{z=0\}$ along $\ell_3^*$. By the maximum principle using vertical planes we deduce that $\ell_ 3^*$ is concave with respect to $\Omega^*$, and $N^*$ coincides with the outer conormal to $\Omega^*$ along $\ell_ 3^*$.
  
  Let $W^*$ be the open convex wedge region bounded by $\{y=0\}$ and $P_{O,p_3,p_4}$ so that $M^*$ is locally contained in $W^*$ in neighborhoods of $\ell_1^*$ and $\ell_2^*$. In particular, $\ell_3^*$ is locally contained in $W^*$ in a neighborhood of $p_3^*$.  
  
  We are going to prove that the whole curve $\ell_3^*$ is contained in $W^*$. Since there are no two points in $\ell_3$ with the same Gauss map and $N$ is orthogonal to $P_{O,p_3,p_4}$ at $p_4$, there cannot be an interior point of $\ell_3^*$ where the maximal distance to $P_{O,p_3,p_4}$ is reached, and $p_4^*$ must be the farthest point from $P_{O,p_3,p_4}$ in $\ell_3^*$. In particular, $\ell_3^*$ lies on one side of $P_{O,p_3,p_4}$. On the other hand, $\ell_3^*$ is strictly concave with respect to $\Omega^*$. Moreover, $\ell_3^*$ has length one and it goes out orthogonally from $P_{O,p_3,p_4}$ at $p_3^*$. Since the segment orthogonal to $P_{O,p_3,p_4}$ from $p_3^*$ to $\{y=0\}$ has length one, we deduce that $\ell_3^*$ cannot intersect $\{y=0\}$, and it is contained in $W^*$.

\begin{figure}[htb!]
  \centering \includegraphics[width=.5\textwidth]{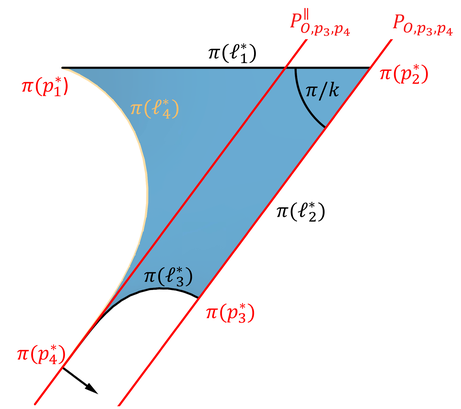}
  \caption{The projection of the surface $M^*$ on $\{z=0\}$.}
  \label{fig:projection_k}
\end{figure}

It just remains to study the behaviour of the curve $\ell_4^*$.  As in the case of the tensile catenoid, $\ell_4^*$ is an asymptotic curve with constant curvature $1$ whose curvature vector coincides with the surface outer conormal along $\ell_4^*$.  From Lemma~\ref{l:Nhoriz_k} we know that the vertical coordinate of $N^*$ along $\ell_4^*$ never vanishes, so the projection of the curve is concave or convex with respect to $\Omega^*$. By Remark~\ref{r:cusp_k}, we deduce that it must be concave.

From the fact that the total variation of $N$ along $\ell_3$ is $\frac \pi 2$ and $\ell_3,\ell_4$ are asymptotic at $p_4$, we get that $N^*(p_4^*)$ is orthogonal to $P_{O,p_3,p_4}$.

We call $\tilde W^*\subset W^*$ the wedge region bounded by the half-planes with boundary the $z$-axis and passing through $p_1^*$ and $p_4^*$, see picture .  Since $N^*(p_4^*)$ arrives orthogonally to the parallel plane to $P_{O,p_3,p_4}$ passing through $p_4^*$, we deduce that $\ell_4^*$ is contained in $\tilde W^*$ near $p_4^*$.

Take a vertical plane $P$ passing through the origin and intersecting $\tilde W^*$. We consider the farthest point $q$ of $M^*\cap P$ from the $z$-axis. We observe that this point cannot be interior to $M^*$, as the Gauss map at that point would be horizontal, in contradiction to the fact that $M^*$ is a vertical multigraph.  Hence $q$ must be a point in $\ell_3^*\cup \ell_4^*$. We observe that $\ell_3^*\cap \tilde W^*$ (if it is non-empty), together with a segment in $\partial \tilde W^*\cap\{z=0\}$, bounds a convex domain in $\{z=0\}$, which we call~$D$. In a neighborhood of $\ell_3^*$, $M^*$ does not project on $D$ ($\ell_3^*$ is concave with respect to $\Omega^*$). We conclude that $q$ cannot be in $\ell_3^*$. Hence $q\in\ell_4^*$. 
We prove, as in the case of the tensile catenoid, that there cannot be more points in $P\cap\ell_4^*$.  Therefore, we obtain that any vertical half-plane in $\tilde W^*$ with the $z$-axis as the boundary necessarily intersects $\ell_4^*$ in exactly one point.

We conclude that $\ell_4^*$ must be embedded with injective vertical projection, and $M^*$ is a vertical graph. 

To finish the proof of the existence of the tensile $k$-noid, we only need to extend $M^*$ by reflection symmetry on the planes containing the boundary curves $\ell_1^*,\ell_2^*$ and $\ell_3^*$. Since $M^*\subset W^*$, it remains to prove that $M^*\subset\{z>0\}$ to obtain that the extended surface is embedded.

By the mean value theorem, a point of $M^*\cap\{z<0\}$ would produce a point where the conormal to the boundary curve $\ell_4^*$ changes sign. As in the proof of the tensile catenoid, the existence of such a point would imply the existence of a point in $M$ where the tangent vector of $\ell_4$. Thus $\ell_4$ would have a local extremum, a contradiction to Lemma \ref{l:rotation}.

Finally, by Remark~\ref{r:cusp_k} we get that the tensile $k$-noid has cusps at the points $p_1^*$, $p_4^*$ and their symmetric ones.
\end{proof}

\bibliographystyle{plain}
\bibliography{Biblio.bib}

\end{document}

%% file: settings.tex
\setlist{nolistsep}

\graphicspath{{figures/}}